\documentclass[12pt,a4paper,reqno]{amsart}

\usepackage{bbm}

\usepackage{amsmath,amssymb,amsthm,amscd,mathrsfs}

\usepackage[english]{babel}
\usepackage[mathscr]{eucal}

\usepackage[pdftex]{color,graphicx}

\usepackage[pdftex,bookmarks=true,colorlinks=true, pdfstartview=FitV, linkcolor=black,  citecolor=black, urlcolor=black,pdfstartview=FitH]{hyperref}
 \usepackage{stmaryrd}
\usepackage{color}
\usepackage[applemac]{inputenc} 
\usepackage{fancyhdr}
\usepackage{epsfig}
\usepackage{cite}
\usepackage{graphicx}
\usepackage{latexsym}
\usepackage{bm}
\usepackage{mathtools}				
\usepackage{esdiff}					
\usepackage{epic}
\usepackage{multicol}
\usepackage{supertabular} 
\usepackage{listings}
\usepackage{verbatim}
\usepackage{turnstile} %

\usepackage{verbatim}

\usepackage{tikz}
\usetikzlibrary{arrows.meta}				
\usetikzlibrary{decorations.markings}		         
\usepackage{tikz-dimline}					

\usepackage{xcolor}

\usetikzlibrary{arrows.meta,calc,decorations.pathreplacing,positioning}

\definecolor{trajblue}{RGB}{35,92,155}
\definecolor{distorange}{RGB}{214,126,48}
\definecolor{regiongray}{RGB}{105,105,105}
\definecolor{stripblue}{RGB}{222,235,247}
\definecolor{cutgreen}{RGB}{61,139,94}
\definecolor{accentred}{RGB}{172,54,54}
\definecolor{shellpurple}{RGB}{112,78,154}

\usepackage{dsfont}

\font\sc=rsfs10 at 12pt
\numberwithin{equation}{section}

\renewcommand{\a}{\alpha}
\renewcommand{\b}{\beta}
\newcommand{\g}{\gamma}
\newcommand{\G}{\Gamma}
\renewcommand{\d}{\delta}
\newcommand{\D}{\Delta}
\newcommand{\e}{\epsilon}
\newcommand{\ve}{\varepsilon} 

\newcommand{\y}{\eta}
\renewcommand{\th}{\theta}

\renewcommand{\l}{\lambda}
\renewcommand{\L}{\Lambda}
\newcommand{\m}{\mu}
\newcommand{\n}{\nu}
\newcommand{\x}{\xi}

\renewcommand{\r}{\rho}
\newcommand{\s}{\sigma}
\renewcommand{\t}{\tau}
\newcommand{\f}{\phi}

\newcommand{\F}{\Phi}
\newcommand{\h}{\chi}
\newcommand{\p}{\psi}

\renewcommand{\o}{\omega}
\renewcommand{\O}{\Omega}

\newcommand{\C}{{\mathbb C}}
\newcommand{\Dl}{{\mathbb D}}

\newcommand{\R}{{\mathbb R}}

\newcommand{\N}{{\mathbb N}}

\newcommand{\Vl}{{\mathbb V}}
\newcommand{\Wl}{{\mathbb W}}

\newcommand{\mbf}[1]{\protect\mbox{\boldmath$#1$\unboldmath}}

\newcommand{\ab}{\mbf{a}}

\newcommand{\bb}{\mbf{b}}
\newcommand{\bt}{{\tilde{\mathbf b}}}
\newcommand{\cb}{\mbf{c}}

\newcommand{\eb}{\mbf{e}}

\newcommand{\rb}{\mbf{r}}

\newcommand{\Ab}{\mbf{A}}

\newcommand{\Jb}{\mbf{J}}

\newcommand{\Mb}{\mbf{M}}

\newcommand{\Rb}{\mbf{R}}

\newcommand{\Ac}{{\mathcal A}}
\newcommand{\Bc}{{\mathcal B}}

\newcommand{\Jc}{{\mathcal J}}

\newcommand{\Oc}{{\mathcal O}}

\newcommand{\Rc}{{\mathcal R}}

\newcommand{\asf}{{\mathsf a}}
\newcommand{\bsf}{{\mathsf b}}

\newcommand{\Ssf}{{\mathsf S}}

\newcommand{\ot}{\otimes}
\newcommand{\supp}{\hbox{{\rm supp}}\,}

\newcommand{\kum}[1]{\langle {#1} \rangle}

\newcommand{\pd}{\partial} 

\DeclareMathOperator{\im}{{\rm Im}\,}
\DeclareMathOperator{\re}{{\rm Re}\,}

\newcommand{\Res}{\operatorname{Res\,}}

\newcommand{\dist}{\operatorname{dist\,}}

\newcommand{\Tr}{\operatorname{Tr\,}}
\newcommand{\tr}{\operatorname{tr\,}}

\newcommand{\1}{\mathbf 1}

\newcommand\Item[1][]{%
	\ifx\relax#1\relax  \item \else \item[#1] \fi
	\abovedisplayskip=0pt\abovedisplayshortskip=0pt~\vspace*{-\baselineskip}}

\newcommand{\bopp}[1]{\mathcal{B}(#1)}

\newcommand{\cs}{C^{\infty}} 
\newcommand{\ccs}{C_{0}^{\infty}} 

\newcommand{\tD}{\sc\mbox{D}\hspace{1.0pt}} 
\newcommand{\tE}{\sc\mbox{E}\hspace{1.0pt}} 
\newcommand{\tS}{\sc\mbox{S}\hspace{1.0pt}} 

\newcommand{\ddist}{\tD^{\prime}} 
\newcommand{\dtemp}{\tS^{\prime}} 
\newcommand{\edist}{\tE^{\prime}} 

\newcommand{\wfa}{\operatorname{WF}_{\operatorname{a}}}

\newcommand{\Fr}{\sc\mbox{F}\hspace{1.0pt}} 

\newcommand{\opc}[1]{{\rm Op}[{#1}]}

\newcommand{\sm}{\operatorname{spec}\,}

\newcommand{\sme}{\operatorname{spec}_{\operatorname{ess}}\,}

\newcommand{\smac}{\operatorname{spec}_{\operatorname{ac}}}

\newcommand{\ssa}{\operatorname{singsupp}_{\text{a}}\,}

\usepackage{bm}

\newcommand{\alb}{\bm{\a}}                 
\newcommand{\beb}{\bm{\b}}                 
\newcommand{\Pib}{\bm{\Pi}}                
\newcommand{\Idb}{\mathbf{I}}              

\newcommand{\Pot}{\Vl}              
\newcommand{\Wpot}{\Wl}             
\newcommand{\dsym}{\mbf{d}}                

\newcommand{\Mfour}{M_4(\C )}
\newcommand{\Sone}{\mathcal{S}_1}
\newcommand{\dLeb}{L}
\DeclareMathOperator{\Opw}{Op^{w}_{\hbar}}

\newtheorem{theorem}{Theorem}[section]
\newtheorem{proposition}[theorem]{Proposition}
\newtheorem{lemma}[theorem]{Lemma}
\newtheorem{corollary}[theorem]{Corollary}

\theoremstyle{definition}
\newtheorem{definition}[theorem]{Definition}

\theoremstyle{remark}
\newtheorem{remark}[theorem]{Remark}
\newtheorem{example}[theorem]{Example}

 \newenvironment{myindentpar}[1]%
 {\begin{list}{}%
         {\setlength{\leftmargin}{#1}}%
         \item[]%
}
 {\end{list}}

\begin{document}

	\tikzset{->-/.style={decoration=
				{
				markings,
				mark=at position #1 with
				{\arrow[line width = 0.6mm]{stealth}}
				},
				postaction={decorate}
				}}


\title[Matrix-valued Weyl trace formulas and Dirac resonances...]{Matrix-valued Weyl trace formulas and \\ 
semiclassical resonances of Dirac operators \\ 
with mass-type potentials}

\author{Zhuo Chen}

\author{Michael Melgaard} 
\address[Z.~Chen,  M.~Melgaard]{Department of Mathematics \\ School of Mathematical  and Physical Sciences \\
                                     University of Sussex      \\
                                     BN1 9QH Brighton, United Kingdom} 
\email{m.melgaard@sussex.ac.uk (corresponding author)}

\date{November 24, 2026}



%
%
\begin{abstract}
We prove three results for semiclassical Dirac operators on $\R^3$. First, we establish a relative semiclassical Weyl-type trace 
formula for a pair of Dirac operators perturbed by \textit{general Hermitian matrix-valued potentials}: 
for $g\in\ccs(\R)$, $\Tr[g(\Dl_2)-g(\Dl_1)]=(2\pi\hbar)^{-3}\int_{\R^6}\tr[g(\dsym_2)-g(\dsym_1)]\,dx\,d\x+\Oc(\hbar^{-2})$,
with an explicit dependence of the remainder on $g$. Only symbol bounds on the potentials and decay of order $\d>3$ 
of their difference are required; no separation of the eigenvalue branches of the symbols and no decay of the individual 
potentials is assumed. This extends to the full matrix-valued setting the formula of Bruneau and Robert for electric perturbations.
Second, for mass-type perturbations $\beb v$ we show that analytic singularities of a relative distribution built from the mass functions $m=1+v$ generate resonances: every boundary point $E_0>1$ of its support attracts at least $C\hbar^{-3}$ resonances, and a single operator
$\Dl_0+\beb v$ with $\sup v>0$ has at least $C\hbar^{-3}$ resonances in every complex neighbourhood of $\pm E_0$, $E_0=1+\sup v$. 
By Khochman's upper bound the power is optimal. Third, we show that resonances are born at trapping energies: $\pm E_0$ always 
carries an equilibrium of the Hamilton flow and admits no escape function, whereas on nontrapping energy intervals beyond $E_0$ the resonances are excluded, uniformly, from a strip of depth $C\hbar\log\hbar^{-1}$. A key point in the second result is that the passage from the mass
functions to the classical density of states is not a convolution; it becomes one after the
substitution $u=E^2$, which allows us to transport analytic wave-front singularities through an
elliptic analytic symbol of order $-3/2$.
\end{abstract}

\maketitle

\section{Introduction}
\label{ZCMM2026d:introduction}
This paper is concerned with semiclassical Dirac operators $\Dl=\Dl_0+\Pot$ on $L^2(\R^3;\C^4)$, 
$\Dl_0=-i\hbar\sum_{k=1}^3\alb_k\partial_k+\beb$, perturbed by matrix-valued potentials $\Pot$, 
and with their resonances. It contains three results of independent interest. Firstly, we establish a 
relative semiclassical Weyl law for full matrix-valued perturbations in Theorem~\ref{ZCMM2026d:weyl_trace_formula}. 
For two Dirac operators with arbitrary Hermitian $4\times4$ matrix-valued potentials whose difference decays faster 
than $\kum{x}^{-3}$, the relative trace $\Tr[g(\Dl_2)-g(\Dl_1)]$ is given to leading order by the phase-space integral of 
the matrix trace of $g(\dsym_2)-g(\dsym_1)$, with remainder $\Oc(\hbar^{-2})$  controlled explicitly by $g$. Second, 
Theorem~\ref{ZCMM2026d:thm_rel} and Theorem~\ref{ZCMM2026d:thm_abs} establish the existence of resonances for 
mass-type  perturbations, with optimal counting.  Analytic singularities of a distribution built from the mass functions $m=1+v$ 
force the existence of at least $C\hbar^{-3}$ resonances; for a single operator they appear near $\pm E_0$, 
$E_0=1+\sup v$. Together with Khochman's upper bound this shows that $\hbar^{-3}$ is the correct order.
Third, we show that resonances are born at trapping energies; see Proposition~\ref{ZCMM2026d:prop_trapping} and Theorem~\ref{ZCMM2026d:cor_wedge}. The energies $\pm E_0$ singled out in Theorem~\ref{ZCMM2026d:thm_abs} are 
necessarily trapping, and no escape function exists there; on nontrapping energy intervals beyond $E_0$,  by contrast, 
resonances are excluded uniformly from a logarithmic strip. We call this the \textit{trapping--nontrapping dichotomy} for 
the resonances of mass-type Dirac operators. Theorem~\ref{ZCMM2026d:weyl_trace_formula} is the analytic tool on which 
Theorem~\ref{ZCMM2026d:thm_rel} rests, but it is valid in much greater generality and we expect it to be useful elsewhere; 
Proposition~\ref{ZCMM2026d:prop_trapping} and Theorem~\ref{ZCMM2026d:cor_wedge} explain the dynamical meaning of 
Theorem~\ref{ZCMM2026d:thm_abs} and relates it to the resonance-free domains of Chen--Melgaard \cite{chen_melgaard26}.

Our first result, Theorem~\ref{ZCMM2026d:weyl_trace_formula}, is a Weyl-type formula. Such formulas go back to Robert 
\cite{robert_trace} for Schr\"odinger operators. For the Dirac operator, Bruneau and Robert \cite[pp.~21--22]{bruneau_robert99} 
obtained such a formula, comparing each operator with $\Dl_0$, for perturbations of the form $-\alb\cdot \Ac+vI_4$ with a 
\textit{scalar} electric potential $v$; their semiclassical analysis rests on a semiclassical Foldy--Wouthuysen decoupling, which 
needs the two eigenvalue branches of the symbol to be smooth and uniformly separated. Theorem~\ref{ZCMM2026d:weyl_trace_formula} 
covers arbitrary Hermitian matrix-valued perturbations, including mass-type (or Lorentz-scalar) potentials $\beb v$, magnetic potentials 
$-\alb\cdot \Ac$, anomalous-magnetic-moment terms $i\beb\,\alb\cdot E$ and all Hermitian combinations of these. It has three further 
features which are essential in the applications below: (i) no separation of the eigenvalue branches of $\dsym_j$ is required; 
(ii) only the \textit{difference} $\Pot_2-\Pot_1$ has to decay, while the individual potentials need only satisfy symbol bounds; 
(iii) the dependence of the remainder on $g$ is explicit, which is what allows us to apply the formula to the families of FBI wave 
packets used in the proof of Theorem~\ref{ZCMM2026d:thm_rel}, our second main result. Our proof of Theorem~\ref{ZCMM2026d:weyl_trace_formula} (see Section~\ref{ZCMM2026d:weyl_section}) is different from that of 
\cite{bruneau_robert99}. It combines the Helffer--Sj\"ostrand formula with one \textit{left} and one \textit{right} parametrix, arranged 
so that the decaying difference $\Pot_2-\Pot_1$ is always adjacent to a smoothing factor, and with a Stokes-type observation: 
although the Dirac resolvent is a symbol of order $-1$ only, after integration against $\bar\partial\tilde g$ the relevant 
symbols are compactly supported in $\x$, so that integrability in the momentum variable comes for free. 
Bruneau--Robert's general set-up allows $V=\mathrm{diag}(V_+I_2,V_-I_2)$ together with a magnetic potential, and 
$\beb v=\mathrm{diag}(vI_2,-vI_2)$ is formally of this form. However, their semiclassical analysis, in which the weak asymptotics 
on pp.~21--22 appear, is carried out only for $V_+=V_-$

Our second result, Theorem~\ref{ZCMM2026d:thm_rel} for a pair of Dirac operators, and also its companion, 
Theorem~\ref{ZCMM2026d:thm_abs} for a single Dirac operator, concern the existence of resonances for mass-type perturbations. 
We show, in the spirit of Sj\"{o}strand \cite{sjostrand97_2} and Nedelec \cite{nedelec01} that analytic singularities of a relative classical density of states force the existence of semiclassical Dirac resonances for 
mass-type perturbations.  Our general result concerns a pair $\Dl_j=\Dl_0+\beb v_j$, $j=1,2$, satisfying the exterior analyticity, 
decay and non-degeneracy conditions of Assumption~$(\Ab _\d )$.  Associated with the two mass functions $m_j=1+v_j$ is the 
relative distribution $\m$ defined in~\eqref{ZCMM2026d:mu_def}. We prove in Theorem~\ref{ZCMM2026d:thm_rel} that if $E_0>1$ 
is any boundary point of $\supp(\m)$, then every complex neighbourhood $U$ of $E_0$ contains
resonances satisfying
\[
   \sum_{j=1}^{2}    \#\bigl(\Res(\Dl_j)\cap U\bigr)   \geq C(U)\hbar^{-3}
\]
for all sufficiently small $\hbar$, and the same conclusion holds near $-E_0$. The corresponding single-operator result is obtained by 
constructing a comparison potential for which the comparison operator has no resonances near the energy in question.  More precisely, 
let $\Dl=\Dl_0+\beb v$,  where $v\in C^\infty(\R^3;\R)$ is bounded, extends holomorphically to an exterior sector with power decay there, and satisfies $\inf_{x\in\R^3}(1+v(x))>0$, $\sup_{x\in\R^3}v(x)>0$. Putting $E_0=1+\sup_{x\in\R^3}v(x)$, Theorem~\ref{ZCMM2026d:thm_abs} gives, for every complex neighbourhood $U$ of $E_0$,
\[
   \#\bigl(\Res(\Dl)\cap U\bigr)    \geq C(U)\hbar^{-3}
\]
for some constant $C(U)$ depending on $U$ and all sufficiently small $\hbar$, and likewise near $-E_0$.  Thus the lower bound has the same semiclassical order as Khochman's $\Oc(\hbar^{-3})$ upper bound. The geometry underlying the single-operator result is encoded by the 
threshold function $E_{\rm thr}(x) :=\min_{\x\in\R^3}p_+(x,\x) =m(x)=1+v(x)$.  The energy $E_0$ is precisely the maximum of 
$E_{\rm thr}$, and the single-operator theorem therefore produces resonances near this maximal threshold; the second family, near 
$-E_0$, is its mirror image under the spectral symmetry of the mass-type Dirac operator. It is worth noting that an electric perturbation $\Pot=vI_4$ has the same positive threshold function, since $\min_{\x}\bigl(v(x)+\kum{\x}\bigr)=1+v(x)$; the mechanism which produces the resonances is, however, quite different, as we explain below.

The proof proceeds along the chain
\begin{align*}
E_0\in\pd\supp(\m),\quad E_0>1
   &\Longrightarrow (E_0,\pm1)\in\wfa(\m)\\
   &\Longrightarrow (E_0,\pm1)\in\wfa(\o)
   \Longrightarrow \#\Res\geq C\hbar^{-3}.
\end{align*}
The first implication is Lemma~\ref{ZCMM2026d:endpoint_lemma}; the second
is the transfer of analytic singularities established below, and the
third follows by combining Khochman's local trace formula with the
relative semiclassical trace formula and a family of FBI wave
packets.  For the single-operator statement,
Theorem~\ref{ZCMM2026d:thm_abs}, the comparison potential constructed in
Appendix~\ref{ZCMM2026d:appendix_comparison} is chosen so that $E_0=1+\sup v=\max\supp(\m)$, and so that the comparison operator has no resonances in a
neighbourhood of $E_0$.  Theorem~\ref{ZCMM2026d:thm_abs} then follows
from Theorem~\ref{ZCMM2026d:thm_rel}.

For Schr\"odinger operators a resonance admits several definitions which look quite different: 
a pole of the meromorphically continued resolvent, an eigenvalue of a dilated \cite{aguilar_combes,balslev_combes} or distorted \cite{hunziker} operator, or a point produced by the microlocal machinery devised for the semiclassical regime by Helffer and Sj\"ostrand \cite{helsjo86}. Whenever two of them apply to the same operator they single out the same 
points, with the same multiplicities, as Helffer and Martinez showed \cite{helmar87}. 
We shall work with the distorted operator throughout. For the general theory, and for a 
systematic account of the counting problem which we take up here, we refer to the monograph of Dyatlov and Zworski \cite{dyatlov_zworski19}. 

The mechanism that we exploit originates with Sj\"ostrand \cite{sjostrand97_2}. For a semiclassical Schr\"odinger operator $-\hbar ^2\D + V$ he observed that his local trace formula 
\cite{sjostrand97}, played off against a Weyl-type formula, 
converts non-analyticity of a certain distribution attached to $V$ 
into the presence of many resonances: each point of the analytic singular support of that distribution attracts of the order of $\hbar ^{-n}$ of them. Nedelec \cite{nedelec01} carried this over to Schr\"odinger operators with matrix-valued potentials, the essential new ingredient being a local trace formula for systems. 
What follows is the analogue for a Dirac operator whose perturbation acts on the mass term. 
As we explain below, the passage from the potential to the classical density of states is 
then no longer given by a convolution, and it is at this point that the mass-type case 
departs from all of the preceding ones. 

For the three-dimensional Dirac operator, resonances were first defined rigorously by means of complex dilation, in the spirit of Aguilar-Balslev-Combes-Simon theory, by Weder \cite{weder73} 
and Seba \cite{seba}. In the semiclassical regime Parisse \cite{parisse91,parisse92} adapted the theory of Helffer and Sj\"ostrand \cite{helsjo86} and located shape resonances exponentially close to the real axis, the imaginary part of the first one being governed by an Agmon-type distance between the
attractive well and the free-particle Dirac sea; Amour, Brummelhuis and Nourrigat
\cite{amour_brummelhuis_nourrigat_01} proved existence in the non-relativistic limit for
potentials growing at infinity; and Ito \cite{ito20} has since treated dilation-analytic
potentials which diverge at infinity. Khochman \cite{khochman07} set up the complex
distortion for smooth matrix-valued potentials admitting an analytic extension into a
sector of $\C ^3$ around $\R ^3$ outside a compact set, with power decay there, and proved
in that setting an Aguilar-Balslev-Combes-Simon theorem, a local trace formula and the
upper bound $\Oc (\hbar ^{-3})$ quoted above. A second, non-semiclassical line of work
defines resonances as poles of the meromorphically continued cut-off resolvent: Kungsman
and Melgaard \cite{jkmm17} established a global Poisson wave trace formula for compactly
supported Hermitian perturbations, Cheng and Melgaard \cite{cheng_melgaard21} extended it
to the spectrum edges $\pm mc^2$ and deduced from it that such perturbations produce
infinitely many resonances, and Dumant \cite{dumant25} has recently proved a Rellich-type
theorem which excludes real resonances, both for compactly supported fields and for the
MIT bag model. In one dimension and in the radial case the counting function itself is
known, by the work of Iantchenko and Korotyaev \cite{iantchenko_korotyaev14}.  Finally, for a \emph{scalar} (electric) perturbation $V=vI_4$ of the semiclassical Dirac operator, 
Kungsman-Melgaard \cite{jkmm14} established the matching lower bound of order
$\hbar ^{-3}$ near $\sup v + 1$ and $\inf v -1$.

Mass-type perturbations, that is perturbations of the form $\beb\, v(x)$, arise in several
distinct physical models: in condensed-matter and metamaterial systems such as graphene,
Dirac materials and topological insulators the low-energy quasi-particles behave like
Dirac fermions with a mass gap that varies in space; in the MIT bag model for hadrons;
and in the Dirac phenomenology of nucleons in a mean field, where the scalar potentials
generated by $\s $-meson fields effectively modify the nucleon mass inside the nucleus.
We refer to the introduction of Chen and Melgaard \cite{chen_melgaard26} for a discussion
of these models and for further references. Mathematically, letting the mass depend on
$x$ changes the geometry of the symbol: the two energy sheets acquire an $x$-dependent
gap, and the threshold at which the branch $p_{+}$ takes its minimum in $\x $ becomes the
function $x\mapsto m(x) = 1+v(x)$ itself. In \cite{chen_melgaard26} an a priori
microlocally exponential weighted estimate on an $I$-Lagrangian manifold was established
for a Dirac operator perturbed by a matrix-valued long-range potential, and it was used to
exhibit a resonance-free domain 
$[\l _0 - \zeta_{0} , \l _0 + \zeta_{0} ] - i[0, C\hbar \log \hbar ^{-1}]$ at a nontrapping energy $\lambda_{0}$ for a Dirac operator with a mass-type potential.

\subsection*{Organisation of the paper}
Section~\ref{ZCMM2026d:notation} fixes notation and assumptions and recalls some analytic microlocal
analysis. Section~\ref{ZCMM2026d:resonances} defines resonances by complex distortion and recalls
Khochman's local trace formula and upper bound. Section~\ref{ZCMM2026d:weyl_section} is devoted to
Theorem~\ref{ZCMM2026d:weyl_trace_formula}. Section~\ref{ZCMM2026d:resonances_near_analytic_sing} 
contains Theorem~\ref{ZCMM2026d:thm_rel} and Theorem~\ref{ZCMM2026d:thm_abs} and their proofs, and Section~\ref{ZCMM2026d:trapping_section} proves Theorem~\ref{ZCMM2026d:cor_wedge}. 
Appendix~\ref{ZCMM2026d:appendix_comparison} constructs the comparison potential.

\section{Notation and assumptions}
\label{ZCMM2026d:notation}
The free, or unperturbed, semiclassical Dirac operator 
\begin{equation}
\Dl_0=-i\hbar\sum_{j=1}^3\alb_j\partial_j+\beb, 
\label{ZCMM2026d:eqD0} 
\end{equation}
initially defined on $C_0^\infty(\R^3;\C^4)$, is essentially self-adjoint. We denote its 
unique self-adjoint closure again by $\Dl_0$; it acts on $L^2(\R^3;\C^4)$ with 
domain $H^1(\R^3;\C^4)$. Here $\alb _j$ and $\beb $ are Hermitian $4\times 4$ matrices
satisfying the usual anti-commutation relations
\begin{equation}
\left \{ \begin{array}{lcl}
\alb _j \alb _k + \alb _k \alb _j = 2\d _{jk} \Idb _4 , & \quad & j,k=1,2,3, \\
\alb _j \beb + \beb \alb _j = 0 , & \quad & j=1,2,3,
\end{array} \right.
\label{ZCMM2026d:anticommutation}
\end{equation}
and $\beb ^2 = \Idb _4$; $\Idb _4$ being the $4\times 4$ identity matrix. It is well-known (see, e.g., \cite{thaller}) that the spectrum of $\Dl _0$ is purely absolutely continuous and equals
$$
\sm (\Dl _0) = \smac (\Dl _0) = (-\infty , -1] \cup [1,\infty ).
$$
Units are chosen so that $c=m=1$; in physical units the thresholds are $\pm mc^2$ and the
energies $E_0$ appearing below read $mc^2 + \sup V$.

We consider a pair $\Dl _j = \Dl _0 + \Pot _j$, $j=1,2$, of perturbations of $\Dl _0$
by \emph{mass-type} potentials
\begin{equation} 
\Pot _j (x) = \beb \, v_j (x), \qquad v_j : \R ^3 \to \R ,
\label{ZCMM2026d:Vmasstype}
\end{equation} 
so that
\begin{equation}
\Dl _j = -i\hbar \sum _{k=1}^3 \alb _k \partial _k + \beb \, m_j(x),
\qquad m_j := 1 + v_j ,
\label{ZCMM2026d:dirac_mass}
\end{equation}
where $m_j$ is the (position dependent) mass function. We assume:

\begin{myindentpar}{1cm}
\textbf{Assumption} $(\Ab _\d )$: $v_j \, : \, \R ^3 \to \R $ is a bounded smooth function
and it has an analytic extension into a sector
$$
C_{\ve , R_0} := \{ z\in \C ^3 \, : \, |\im z| \le \ve |\re z| , \, |\re z| > R_0 \}
$$
for some $\ve \in (0,1)$ and $R_0\ge 0$. We assume that for some $\r >0$,
$$
|v_j(z)| = \Oc (\kum{z}^{-\r }), \qquad z \in C_{\ve ,R_0},\qquad \kum{z} := (1 + |z|^2)^{1/2}, 
$$
that, for some $\d >3$,
$$
|v_2(z) - v_1(z)| = \Oc (\kum{z}^{-\d }), \qquad z\in C_{\ve , R_0},
$$
and, finally, that the mass does not degenerate:
$$
c_0 := \min _{j=1,2}\, \inf _{x\in \R ^3} m_j(x) > 0 .
$$
\end{myindentpar}

The last requirement, absent in \cite{jkmm14}, is what makes the two branches of the
symbol uniformly separated; it holds automatically in the perturbative regime
$\| v_j \| _\infty < 1$, i.e.\ $\| V \|_\infty < mc^2$ in physical units, which covers all
the models mentioned in Section~\ref{ZCMM2026d:introduction}. Note also that
Assumption~$(\Ab _\d )$ contains no hypothesis on the derivatives of $v_j$.

\begin{lemma}
Let $v$ satisfy the analyticity and decay hypotheses of Assumption $(\Ab _\d )$. Then there are constants $C_\a >0$ such that
$$
|\partial ^\a v(x)| \le C_\a \kum{x}^{-\r - |\a |}, \qquad x \in \R ^3 ,
$$
and, more generally, $|\partial ^\a v(z)| \le C_\a \kum{z}^{-\r -|\a |}$ for
$z\in C_{\ve /2, 2R_0+1}$. The same holds for $v_2 - v_1$, with $\r $ replaced by $\d $.
\label{ZCMM2026d:cauchy_estimates}
\end{lemma}

\begin{proof}
Let $|\re z| \ge 2R_0+1$ and $|\im z| \le \frac{\ve }{2}|\re z|$. The polydisc of radius
$r=\frac{\ve }{4}|\re z|$ about $z$ is contained in $C_{\ve , R_0}$: its points $z'$
satisfy $|\im z'| \le \frac{\ve }{2}|\re z| + r \le \ve (|\re z| - r) \le \ve |\re z'|$
and $|\re z'| \ge |\re z| - r > R_0$. Since $|v| \le C\kum{z'}^{-\r } \le C'\kum{z}^{-\r }$
there, Cauchy's inequalities give
$|\partial ^\a v(z)| \le C'\a ! \, r^{-|\a |}\kum{z}^{-\r } \le C_\a \kum{z}^{-\r -|\a |}$.
On the remaining bounded set the derivatives are bounded by smoothness.
\end{proof}

By introducing the semiclassical Fourier transform
\begin{equation} 
(\Fr _\hbar u)(\x ) = \frac{1}{(2\pi \hbar )^{3/2}}\int \limits _{\R ^3}u(x)
e^{-ix\cdot \x /\hbar } \, dx
\label{ZCMM2026d:Ftrans} 
\end{equation}
we can express $\Dl _j$ as $\hbar $-pseudodifferential operators 
$\Dl_j = \mathrm{Op}^w_\hbar(\dsym_j)$,
where the matrix-valued symbols
\begin{equation}
\dsym _j (x,\x ) = \sum _{k=1}^3 \alb _k \x _k + \beb \, m_j(x)
\label{ZCMM2026d:symbol}
\end{equation}
have two-fold degenerate eigenvalues, as we now record; the computation is that of
\cite[Lemma 5.1]{chen_melgaard26}.

\begin{lemma}
Let $\dsym _j$ be given by \eqref{ZCMM2026d:symbol} and put
$$
\t _j (x,\x ) := \big ( m_j(x)^2 + |\x |^2 \big )^{1/2}.
$$
Then $\dsym _j ^2 = \t _j^2 \Idb _4$ and $\dsym _j$ has the two eigenvalues
\begin{equation}
p_{\pm , j}(x,\x ) = \pm \t _j (x,\x ),
\label{ZCMM2026d:eigenvalues_symbol}
\end{equation}
each of multiplicity two, with spectral projectors
$$
\Pib _{\pm , j} = \tfrac12 \big ( \Idb _4 \pm \t _j^{-1}\dsym _j \big ) ,
\qquad \tr \Pib _{\pm ,j} = 2 .
$$
Consequently, for every $g\in \cs (\R )$,
\begin{equation}
\tr \, g(\dsym _j (x,\x )) = 2 \big [ g(\t _j) + g(-\t _j) \big ] ,
\label{ZCMM2026d:trace_of_g}
\end{equation}
and $g(\dsym _j)$ is a smooth function of $(x,\x )$ which vanishes identically for
$|\x | > R$ whenever $\supp g \subset [-R,R]$.
\label{ZCMM2026d:symbol_eigenvalues}
\end{lemma}

\begin{proof}
By \eqref{ZCMM2026d:anticommutation},
$$
\dsym _j^2 = \Big ( \sum _k \alb _k \x _k \Big )^2
+ m_j \sum _k \big ( \alb _k \beb + \beb \alb _k\big ) \x _k + m_j^2 \beb ^2
= \big ( |\x |^2 + m_j^2 \big ) \Idb _4 .
$$
Thus $\dsym _j$ is Hermitian with square $\t _j^2 \Idb _4$, so
$\sm \dsym _j \subset \{ \pm \t _j \}$. Moreover $\tr \dsym _j = 0$: from
\eqref{ZCMM2026d:anticommutation} and $\beb ^{-1}=\beb $ we get
$\tr \alb _k = \tr (\beb \alb _k \beb ) = -\tr \alb _k$, and likewise
$\tr \beb = \tr (\alb _1 \beb \alb _1) = -\tr \beb $. Hence the two eigenvalues have
equal multiplicities, i.e. two each, and
$\tr \Pib _{\pm ,j} = \frac12 (4 \pm \t _j^{-1}\tr \dsym _j) = 2$, which gives
\eqref{ZCMM2026d:trace_of_g} through
$g(\dsym _j) = g(\t _j)\Pib _{+,j} + g(-\t _j)\Pib _{-,j}$. Writing
$$
g(\dsym _j) = \frac{g(\t _j) + g(-\t _j)}{2}\, \Idb _4
+ \frac{g(\t _j)-g(-\t _j)}{2\t _j}\, \dsym _j
$$
and noting that both scalar coefficients are even and smooth functions of $\t _j$, hence
smooth functions of $\t _j^2 = m_j^2 + |\x |^2$ by Whitney's theorem, shows that
$g(\dsym _j) \in \cs (\R ^6)$. Finally $\t _j \ge |\x | > R$ forces
$g(\pm \t _j)=0$.
\end{proof}

We note the elementary but useful consequence of \eqref{ZCMM2026d:eigenvalues_symbol} that,
for $E>0$, only the branch $p_{+,j}$ can contribute to the classical density of states,
since $p_{-,j} \le -c_0 <0$; and symmetrically for $E<0$. 
This is what removes the case distinction which was needed  in \cite{jkmm14}.

\subsection{Analytic microlocal preliminaries} 
\label{ZCMM2026d:analytic_symbols}
This subsection introduces the analytic microlocal terminology in 
Section~\ref{ZCMM2026d:resonances_near_analytic_sing}. 
Everything here is classical: our sources are Sj\"ostrand \cite[Sections 6,7]{sjstrand82} for the wavefront set,
H\"ormander \cite[Chapter 8]{hr_1}, and Treves \cite[Chapter 5]{treves} for the symbol
classes. The semiclassical parameter plays no role and is set equal to one throughout the
subsection. In the sequel, $\ddist(\R^{n})$, $\dtemp(\R^{n})$, and $\edist(\R^{n})$ denote, respectively, the spaces of distributions,  tempered distribution, and compactly supported distributions on $\R^{n}$. 

\smallskip
Let $u \in \dtemp (\R ^n)$, let $(x_0,\x _0) \in \R ^n \times (\R ^n\setminus \{ 0\})$ and let $\h \in \ccs (\R ^n)$ be equal to one near $x_0$. The Fourier-Bros-Iagolnitzer, or FBI,
transform of $u$ is
\begin{equation}
(T_\l u)(x,\x ) = 2^{-n/2}\Big ( \frac{\l }{\pi }\Big )^{3n/4}
\int _{\R ^n}e^{i\l (x-y)\cdot \x - \l (x-y)^2/2}\, \h (y)u(y)\, dy ,
\label{ZCMM2026d:FBI}
\end{equation} 
a family of entire functions indexed by the large parameter $\l $. One says that $u$ is
microlocally exponentially small at $(x_0,\x _0)$ if $T_\l u = \Oc (e^{-C\l })$ for some
$C>0$, uniformly for $(x,\x )$ in a neighbourhood of $(x_0,\x _0)$ and as
$\l \to \infty $. The points at which this fails make up the analytic wavefront set.

\begin{definition}
The analytic wavefront set of $u \in \dtemp (\R ^n)$, denoted $\wfa (u)$, is the
complement in $\R ^n \times (\R ^n \setminus \{ 0\})$ of the set of points where $u$ is of
microlocal exponential decay.
\end{definition}

It is well known that $\wfa (u)$ is a closed conic subset of
$\R ^n \times (\R ^n \setminus \{0\})$ and that the image under the projection onto the
first coordinate equals $\ssa (u)$, i.e.\ the smallest closed set outside of which $u$ is
real analytic (see e.g.\ Sj\"{o}strand \cite[Section 6]{sjstrand82}). If $u$ is real
valued then $\wfa (u)$ is symmetric with respect to $\x \mapsto -\x $, and $\wfa $ is
invariant under real-analytic diffeomorphisms: if $q$ is one, and $q_*$ denotes the push
forward defined by $\langle q_*u , \h \rangle = \langle u , \h \circ q \rangle $, then
\begin{equation}
\wfa (q_*u) = \big \{ (q(t), \y /q'(t)) \, : \, (t,\y ) \in \wfa (u) \big \}
\label{ZCMM2026d:wf_invariance}
\end{equation}
(see \cite[Theorem 8.5.1$'$]{hr_1}).
\\
\\
Next we introduce a certain Gevrey class of symbols, namely the analytic one (see, e.g.,
Treves \cite[Chapter 5]{treves}).

\begin{definition}
A function $a\in C^\infty (\R ^n \times (\R ^n \setminus \{ 0\}))$ is said to belong to
the space $\Ssf_{\rm a}^m(\R ^n)$ of classical analytic symbols if for any
$K \Subset \R ^n$
$$
|\partial _x^\a \partial _\x ^\b a(x,\x )| \le C^{|\a | + |\b | + 1}\a ! \b !
(1+|\x |)^{m-|\b |}
$$
for $x\in K$ and $|\x | \ge B$, where $B$ and $C$ are positive constants.
\label{ZCMM2026d::symbolclass}
\end{definition}

If $a\in \Ssf_{\rm a}^m(\R ^n)$ and $(x_0,\x _0)\in T^*\R ^n\setminus 0$, we say that
$a$ is elliptic at $(x_0,\x _0)$ if there are a neighbourhood $U$ of $x_0$, a conic
neighbourhood $\G _0$ of $\x _0$, and constants $C_0,C_1>0$ such that
$$
|a(x,\x )| \ge C_0 \kum{\x }^m
\quad \text{for }x\in U,\ \x \in \G _0,\ |\x |\ge C_1.
$$
We say that $a$ (and the corresponding analytic pseudodifferential operator) is elliptic
if it is elliptic at every point of $T^*\R ^n\setminus 0$.
\begin{proposition}
Let $A$ be an analytic pseudodifferential operator of order $m$ on $\R ^n$, in the
standard sense of the analytic pseudodifferential calculus, and let $u\in \edist (\R ^n)$, the space of compactly supported distributions. Then
\begin{equation}
\wfa (Au) \subset \wfa (u).
\label{ZCMM2026d:analytic_pseudolocal}
\end{equation}
Moreover, if $A$ is elliptic at $(x_0,\x _0)\in T^*\R ^n\setminus 0$, then
\begin{equation}
(x_0,\x _0)\in \wfa (Au)
\quad \Longleftrightarrow \quad
(x_0,\x _0)\in \wfa (u).
\label{ZCMM2026d:analytic_elliptic}
\end{equation}
Consequently, if $A$ is elliptic on $T^*\R ^n\setminus 0$, then
\begin{equation}
\wfa (Au)=\wfa (u).
\label{ZCMM2026d:analytic_global_elliptic}
\end{equation}
\label{ZCMM2026d:elliptic_wavefront}
\end{proposition}

This is the analytic pseudolocality and elliptic regularity theorem for pseudodifferential
operators; see, for example, Sj\"{o}strand \cite[Sections 6,7]{sjstrand82} and H\"{o}rmander
\cite[Chapter 8]{hr_1}. Notice that the conclusion is a statement about an \emph{analytic pseudodifferential
operator}, not merely about an arbitrary $C^\infty$ completion, for bounded $|\x |$, of a
symbol satisfying the estimates in Definition~\ref{ZCMM2026d::symbolclass}.

We shall use Proposition~\ref{ZCMM2026d:elliptic_wavefront} only for $n=1$, in Lemma~\ref{ZCMM2026d:structure_lemma}. There 
the relevant convolution operator is split into
a high-frequency part and a low-frequency part. The high-frequency multiplier is an
elliptic classical analytic symbol of order $-3/2$, while the inverse Fourier transform of
the low-frequency multiplier is entire; hence the latter part is analytic-regularising on
compactly supported distributions.

Finally we record the elementary observation which will produce the analytic singularity
that drives the whole argument.

\begin{lemma}
Let $u \in \ddist (\R )$ be real valued and let $x_0 \in \supp (u)$ be a point which is
not an interior point of $\supp (u)$. Then $(x_0,\pm 1)\in \wfa (u)$.
\label{ZCMM2026d:endpoint_lemma}
\end{lemma}

\begin{proof}
Since $x_0$ belongs to $\supp (u)$ but is not an interior point of it, every neighborhood
of $x_0$ meets the open set $\R \setminus \supp (u)$ and therefore contains a non-empty
open interval on which $u$ vanishes. If $u$ were real analytic near $x_0$, it would be
given by a convergent power series on some neighborhood $N$ of $x_0$; taking such an
interval inside $N$, analytic continuation would force $u \equiv 0$ on $N$, contradicting
$x_0 \in \supp (u)$. Thus $x_0 \in \ssa (u)$, i.e.\ $(x_0,\y ) \in \wfa (u)$ for some
$\y \ne 0$; since $u$ is real, both signs occur.
\end{proof}

This is the one-dimensional case of \cite[Corollary 8.4.16]{hr_1}, which asserts that if
$u \in \ddist (X)$, where $X$ is an interval on $\R $, and if $x_0 \in X$ is a boundary
point of $\supp u$, then $(x_0,\pm 1)\in \wfa (u)$; we have included the short proof for
completeness.

\section{Resonances}
\label{ZCMM2026d:resonances}
We use the method of complex distortion to define resonances. Deforming the underlying space in a complex direction is an idea whose analytic-dilation form is due to Aguilar and Combes \cite{aguilar_combes} and to Balslev and Combes \cite{balslev_combes}, and whose place in perturbation theory was settled by Simon \cite{simon73}; the flexible version, in which the deformation is generated by a vector field vanishing near the origin and only the behaviour at infinity matters, is Hunziker's \cite{hunziker}. It is the implementation of the latter for the Dirac operator, due to Khochman \cite{khochman07} and used again in \cite[Section 3]{chen_melgaard26}, that we follow here. Once the resonances have been
defined we record the two trace formulas which are played off against each other in
Section~\ref{ZCMM2026d:resonances_near_analytic_sing}.

\subsection{Definition of resonances}
In this subsection $\Dl = \Dl _0 + \Vl$, $\Vl=\beb v$, denotes any perturbation with $v$ satisfying the individual hypotheses of Assumption~$(\Ab _\d )$.

Let $R_0 \ge 0$ be as in Assumption $(\Ab _\d )$ and $\varkappa :\R ^3 \to \R ^3$ be a smooth
function such that $\varkappa (x)=0$ for $|x| \le R_0$ and $\varkappa(x)=x$ outside a compact set
containing $B(0,R_0)$, which also satisfies
$\sup _{x\in \R ^3}\| \nabla \varkappa(x)\| \le \sqrt{2}$. Next introduce
$\f _\th (x) = x + \th \varkappa(x)$ and let $\Jc_\th $ denote the Jacobian determinant of
$\f _\th $. For $\th \in \R $ we define a one-parameter family of distortions on
$\tS (\R ^3;\C ^4)$ by
$$
(U_\th f)(x) = \Jc_\th ^{1/2}(x) f(\f _\th (x)).
$$
When $|\th | < 2^{-1/2}$ it extends to a unitary operator on $L^2(\R ^3;\C ^4)$ (see
\cite{khochman07} for a proof). To define $U_\th $ also for complex values of $\th $ one
first introduces the linear space $\Ac $ of entire functions $f=(f_1,\ldots ,f_4)$ such
that
\begin{align*}
\lim _{\substack{|z| \to \infty \\ z\in C_{\ve ,R_0}}}|z|^k |f_j(z)| = 0,
\quad \text{for } 1\le j \le 4, \quad k \in \N ,
\end{align*}
and let the dense subspace $\Bc $ of so called analytic vectors consist of the functions
in $L^2(\R ^3;\C ^4)$ that have extensions in $\Ac $. With
$$
D_\ve := \big \{ z\in \C \, : \, |z| \le \frac{\ve }{\sqrt{1+\ve ^2}} \big \}
$$
it then holds that for any $f \in \Bc $ the map $\th \mapsto U_\th f$ is analytic for
$\th \in D_\ve $ and $U_\th \Bc $ is dense in $L^2(\R ^3;\C ^4)$ for any
$\th \in D_\ve $. One can then show that for $\th \in D_\ve $ the operator
$$
U_\th \Dl  U_\th ^{-1} = U_\th \Dl _0 U_\th ^{-1} + U_\th \Pot  U_\th ^{-1}
= \Dl _{0,\th } + \Pot  \circ \f _\th
$$
with domain $H^1(\R ^3;\C ^4)$ is an analytic family of type $A$ in the sense of Kato (see
\cite[Chapter 7, Section 2]{kato}). It is shown in \cite{khochman07} that
\begin{equation}
\sme (\Dl _\th ) = \sm (\Dl _{0,\th }) = \G _\th =
\Big \{ z = \pm \Big ( \frac{s }{(1+\th )^2} + 1 \Big )^{1/2} , \, s \in [0,\infty )
\Big \} .
\label{ZCMM2026d:gamma_theta}
\end{equation}
Since $\Pot  = \beb v$ is a smooth Hermitian matrix-valued potential with an analytic
extension into $C_{\ve ,R_0}$ and power decay there, the results of \cite{khochman07},
which are proved in that generality, apply verbatim in the present setting.

The following version of the Aguilar-Balslev-Combes-Simon theorem for the perturbed Dirac
operator was established by Khochman \cite{khochman07}.

\begin{proposition}
For $\th _0 \in D_\ve ^+ = D_\ve \cap \{ \im z \ge 0 \}$ we have
\begin{itemize}
\item[(i)] For $f,g \in \Bc $, the function
$$
F_{f,g}(z) = \langle f, (\Dl - z)^{-1}g\rangle
$$
has a meromorphic extension from
$$
\Sigma = \{ \im z \ge 0, \, \re z > -1 \} \cup \{ \im z \le 0 , \, \re z <1 \}
\setminus \sm (\Dl )
$$
across $\sm (\Dl )$ and into
$$
S_{\th _0} = \Big \{ \bigcup _{\th \in D_\ve ^+}\G _\th \, : \,
\arg (1+\th ) < \arg (1 + \th _0), \, \frac{1}{|1+\th |} < \frac{1}{|1+\th _0|}\Big \} .
$$
\item[(ii)] The poles of the continuation of $F_{f,g}$ into $S_{\th _0}$ are the
eigenvalues of $\Dl _{\th _0}$.
\item[(iii)] These poles are independent of the family $U_{\th _0}$.
\item[(iv)] The operator $\Dl _{\th _0}$ has no discrete spectrum in $\Sigma $.
\end{itemize}
\label{ZCMM2026d:khochmanprop}
\end{proposition}

Proposition~\ref{ZCMM2026d:khochmanprop} justifies the following definition.

\begin{definition}
The resonances of $\Dl $ in $S_{\th _0}\cup \R $, denoted $\Res (\Dl )$, are the
eigenvalues of $\Dl _{\th _0}$. If $z_0$ is a resonance we take its multiplicity to be
the rank of the projection
$$
\frac{1}{2\pi i}\int \limits _{\g _{z_0}} (\Dl _{\th _0} - z)^{-1}\, dz ,
$$
where $\g _{z_0}$ is a sufficiently small positively oriented circle about $z_0$.
\label{ZCMM2026d:definition_resonances}
\end{definition}

Two consequences of this construction will be used. The first is an estimate for the
resolvent of the distorted free operator in the case $R_0 = 0$, in which the distortion
may be taken to be the dilation $\f _\th (x) = (1+\th )x$; it is uniform in $\hbar $,
because the corresponding symbol does not involve $\hbar $. We emphasise that
$\Dl _{0,\th }$ is \emph{not} a normal operator for $\th \notin \R $, so that its
resolvent norm is not the inverse distance to its spectrum; the Clifford identity
\eqref{ZCMM2026d:clifford} replaces normality.

\begin{lemma}
Let $R_0=0$ and $\varkappa(x)=x$, so that $\Dl _{0,\th }$ is the Fourier multiplier with symbol
$\Mb (\x ) = a \sum _k \alb _k \x _k + \beb $, $a_{\th} = (1+\th )^{-1}$. Then
\begin{equation}
\Mb (\x )^2 = \tau_{0,\th} (\x )^2 \Idb _4 , \quad \tau_{0,\th} (\xi)= \big (a_{\th}^2|\x |^2+1\big )^{1/2}, \qquad (\Mb - z)^{-1} = \frac{\Mb + z}{\tau_{0,\th}(\x)^2 - z^2},
\label{ZCMM2026d:clifford}
\end{equation}
and $\{ \pm \tau_{0,\th}(\x ) \, : \, \x \in \R ^3 \} = \G _\th $. Moreover, for $z$ in a fixed
compact subset $K \subset \C \setminus \G _\th $,
$$
\big \| (\Dl _{0,\th } - z)^{-1} \big \| \le \frac{C_{\th , K}}{\dist (z, \G _\th )} .
$$
\label{ZCMM2026d:resolvent_lemma}
\end{lemma}

\begin{proof}
By \eqref{ZCMM2026d:anticommutation},
$\Mb ^2 = a_{\th}^2|\x |^2 \Idb _4 + a_{\th} \sum _k (\alb _k\beb + \beb \alb _k)\x _k + \beb ^2
= (a_{\th}^2|\x |^2+1)\Idb _4$; since $\Mb $ commutes with $z$, the displayed inverse follows,
and the identification of $\{ \pm \tau_{0,\th}(\x )\}$ with $\G _\th $ is
\eqref{ZCMM2026d:gamma_theta} with $\l = |\x |^2$. Consequently
$\| (\Mb - z)^{-1}\| \leq (\| \Mb \| + |z|)/(| \tau_{0,\th}(\x) -z|| \tau_{0,\th}(\x) +z|)$ with
$\| \Mb \| \leq |a_{\th} || \x | + 1 \leq C(| \tau_{0,\th}(\x) |+1)$. 
For $|\x |$ large the numerator is $\Oc (| \tau_{0,\th}(\xi)|)$ while the denominator is 
$\gtrsim | \tau_{0,\th}(\x) |^2$, so the quotient tends to $0$; on the remaining
compact range of $\x $ at most one of the two factors $| \tau_{0,\th}(\x) \mp z|$ can 
be small, and the small one is bounded below by $\dist (z,\G _\th )$. 
Taking the supremum over $\x $ concludes the proof.
\end{proof}

The second consequence is the symmetry of the resonance set, which is special to
mass-type perturbations.

\begin{lemma}
Let $\Jb := \alb _1 \alb _2 \alb _3 \beb $. Then $\Jb $ is invertible and anticommutes
with each of $\alb _1,\alb _2,\alb _3, \beb $. Consequently
$$
\Jb \Dl _j \Jb ^{-1} = -\Dl _j , \qquad \Jb \Dl _{j,\th }\Jb ^{-1} = -\Dl _{j,\th },
$$
and therefore
$$
\Res (\Dl _j) = -\Res (\Dl _j)
$$
with equal multiplicities; the family of resonances near a point $-E_0<-1$ lies in the
upper half plane.
\label{ZCMM2026d:susy_lemma}
\end{lemma}

\begin{proof}
In the Clifford algebra generated by the four anticommuting involutions
$\alb _1,\alb _2,\alb _3,\beb $, the volume element $\Jb $ satisfies $\Jb e = (-1)^3 e\Jb$
for each generator $e$, because $e$ has to be moved past the three remaining generators.
Hence $\Jb $ anticommutes with $-i\hbar \sum _k \alb _k \partial _k$ and with
$\beb m_j(x)$, the scalar function $m_j$ commuting with $\Jb $; thus
$\Jb \Dl _j \Jb ^{-1} = - \Dl _j$. Since $\Jb $ is a constant matrix it commutes with
$U_\th $, which acts by multiplication by a scalar function composed with a change of
variables, and the same identity holds for $\Dl _{j,\th }$. As $\G _\th $, hence
$S_{\th _0}$, is symmetric under $z\mapsto -z$ by \eqref{ZCMM2026d:gamma_theta}, the assertion
follows from Definition~\ref{ZCMM2026d:definition_resonances}.
\end{proof}

\begin{remark}
The chirality matrix alone does not produce this symmetry: in the standard representation
it anticommutes with $\beb $ but commutes with the $\alb _k$, so that conjugation by it
maps $\Dl _j$ to $-i\hbar \sum _k \alb _k\partial _k - \beb m_j(x) \ne -\Dl _j$.
\end{remark}

\subsection{Sj\"{o}strand type trace formula}
\label{ZCMM2026d:trace_formulas}
Khochman \cite{khochman07} establishes the following local trace formula in the spirit of Sj\"{o}strand \cite{sjostrand97}; 
Sj\"{o}strand \cite{sjostrand97_2} and Nedelec \cite{nedelec01} studied the case of Schr\"{o}dinger type operators. See also \cite{bruneau_petkov03} and \cite{nedelec01} for similar results. To state it we make the following assumption on
$\O \subset \C $:
\newline
\newline
\noindent
\textbf{Assumption} $(\Ab _\O ^\pm )$: $\O $ is an open, simply connected and relatively
compact subset of $\{ z \in \C :|\re z|>1\}$ such that $\O \cap \C _\pm \ne \emptyset $ and there exists
$\th _0 \in D_\ve ^+$ such that $\overline{\O } \cap \G _{\th _0} = \emptyset $.
\newline

\noindent
Then one has:

\begin{theorem}
Assume $\O \subset \C $ satisfies ${\rm (}\Ab _\O ^\pm {\rm )}$ and assume in addition that $I:=\O \cap \R $ is an interval.
Let $f$ be holomorphic in a neighbourhood of $\overline\Omega$ and $\h \in \ccs (\R )$ (independent of $\hbar $) be such that (with $d(I,E)$ denoting the distance from $I$ to a point $E$)
\begin{equation*}
\h (E) = \left \{ \begin{array}{ll}
0, & d(I,E) > 2\e_{\star} , \\
1, & d(I,E) < \e_{\star} ,
\end{array} \right.
\end{equation*}
with $\e_{\star} >0$ sufficiently small. Then, for $\Pot _j$ satisfying Assumption
${\rm (}\Ab _\d {\rm )}$ with $\d >3$, one has
\begin{equation*}
 \Tr \big [(\h f)(\Dl _2) - (\h f)(\Dl _1)\big ] = \sum _{z_j \in \Res (\Dl _2)\cap \O } f(z_j)
- \sum _{z_j \in \Res (\Dl _1)\cap \O } f(z_j) + E_{\O , f, \h }(\hbar ),
\end{equation*} 
where
\begin{equation*}
|E_{\O , f , \h }(\hbar )| \le C(\h , \O ) \sup \{ |f(z)| \, : \,
d(\pd \O , z) < 2\e_{\star} , \, \im z \le 0 \} \hbar ^{-3}.
\end{equation*}
\label{ZCMM2026d:local_trace_formula}
\end{theorem}

We shall also need the accompanying upper bound of \cite{khochman07}, which is used in an essential way in Step 5 of the proof of Theorem~\ref{ZCMM2026d:thm_rel} and which we therefore state separately.

\begin{theorem}
Under the hypotheses of Theorem~\ref{ZCMM2026d:local_trace_formula}, one has
\begin{equation}
\# \big ( \Res (\Dl _j) \cap \O \big ) = \Oc (\hbar ^{-3}), \qquad j=1,2 .
\label{ZCMM2026d:upper_bound}
\end{equation}
\label{ZCMM2026d:upper_bound_thm}
\end{theorem}

Thus the lower bound proved in this paper is established by \emph{using} the matching upper bound; the two results are not
 independent of each other.

\section{A relative Weyl law for Dirac operators with matrix-valued potentials}
\label{ZCMM2026d:weyl_section}
In this section we prove the relative semiclassical Weyl law for a pair of Dirac operators
$$
\Dl_j=\Dl_0+\Pot_j=\Opw(\dsym_j),\qquad \dsym_j(x,\x)=\sum_{k=1}^3\alb_k\x_k+\beb+\Pot_j(x),\qquad j=1,2,
$$
perturbed by Hermitian matrix-valued potentials $\Pot_j\in\cs(\R^3;\Mfour)$. The result is independent of the 
rest of the paper and does not use the mass-type structure; in Section~\ref{ZCMM2026d:resonances_near_analytic_sing} 
it is applied with $\Pot_j=\beb v_j$.

\begin{theorem}
Let $\Pot _1, \Pot _2 \in \cs (\R ^3 ; \Mfour )$ be Hermitian and such that
$|\partial ^\a \Pot _j| \le C_\a \kum{x}^{-|\a |}$ and
$|\partial ^\a (\Pot _2 - \Pot _1)| \le C_\a \kum{x}^{-\d -|\a |}$ for all $\a $, with
$\d >3$. Then, for every $g \in \ccs (\R )$, the operator $g(\Dl _2) - g(\Dl _1)$ is of
trace class, with $\| g(\Dl _2) - g(\Dl _1)\| _{\Sone } \le C(g)\hbar ^{-3}$, and
\begin{equation}
\Tr \big [g(\Dl _2) - g(\Dl _1) \big ]  = (2\pi \hbar )^{-3}
\int \limits _{\R ^6} \Big ( \tr [g(\dsym _2)] - \tr [g(\dsym _1)] \Big )\, dx \, d\x + \Oc \big ( C(g)\hbar ^{-2}\big ),
\label{ZCMM2026d:robert_trace_formula}
\end{equation}
where ``$\tr $'' denotes the matrix trace and $C(g)$ is controlled by finitely many
seminorms of $g$ together with $\sup \{ |s| \, : \, s \in \supp g \}$.
\label{ZCMM2026d:weyl_trace_formula}
\end{theorem}

Before proving the theorem we make three remarks. 

\begin{remark}
The explicit dependence of the error term in \eqref{ZCMM2026d:robert_trace_formula} on $g$
matters: in the proof of Theorem~\ref{ZCMM2026d:thm_rel} the formula is applied to a family of
FBI wave packets whose seminorms grow, and it is essential that the packet be fixed before
$\hbar$ is allowed (or taken) to tend to zero.
\label{ZCMM2026d:remark_gdep}
\end{remark}

\begin{remark}[Scope]
\label{ZCMM2026d:remark_scope}
The hypotheses are satisfied by any Hermitian combination of the following perturbations, with smooth bounded 
coefficients satisfying symbol estimates and differences decaying like $\kum{x}^{-\d}$, $\d>3$: electric potentials 
$vI_4$; mass-type (Lorentz-scalar) potentials $\beb v$; magnetic potentials $-\alb\cdot A$; anomalous-magnetic-moment terms 
$i\beb\,\alb\cdot E$. None of the last three commutes with the free symbol $\alb\cdot\x+\beb$. Since only the difference 
$\Pot_2-\Pot_1$ has to decay, the formula also compares two operators which share a non-decaying background, for example 
two perturbations of a fixed bounded smooth Hermitian field. Finally, no separation of the eigenvalue branches of $\dsym_j$ is 
required: the branches of $\alb\cdot\x+\beb+\Pot_j(x)$ may cross, as happens for electric potentials with $V_+\ne V_-$ or for
mass functions with zeros.
\end{remark}

\begin{remark}[Relation to Bruneau--Robert]
\label{ZCMM2026d:remark_BR}
For $\Pot_j=-\alb\cdot A+v_jI_4$ with scalar $v_j$ decaying faster than $\kum{x}^{-3}$, the formula 
\eqref{ZCMM2026d:robert_trace_formula} follows from the weak asymptotics in \cite[pp.~21--22]{bruneau_robert99}, 
applied to each pair $(\Dl_j,\Dl_0)$ and subtracted. Their argument uses the Helffer--Robert functional calculus and, 
in the semiclassical regime, a Foldy--Wouthuysen decoupling of the two eigenvalue branches. The proof 
given below avoids both: it works directly with the relative resolvent identity and therefore needs neither decay of 
$\Pot_j$ relative to $\Pot_0=0$ nor branch separation.
\end{remark}

As mentioned, nothing in the proof of Theorem~\ref{ZCMM2026d:weyl_trace_formula} uses the mass-type structure. What is used is that  $\Pot _1$ and $\Pot _2$ are Hermitian matrix-valued multiplication potentials with the symbol bounds stated in that theorem, 
that the Dirac symbols $\dsym_j(x,\x)=\sum_k\alb_k\x_k+\beb+\Pot_j(x)$ are elliptic of order one 
(which is automatic from their Dirac principal part), and that $\Wpot:=\Pot_2-\Pot_1$ has $x$-decay of order $\d>3$. 
No separation of the eigenvalue branches of $\dsym_j$ is required. 
The trace appearing below is always the relative trace $\Tr\big(g(\Dl_2)-g(\Dl_1)\big)$; the individual operators $g(\Dl_j)$ need not be trace class, so a literal difference of their separate traces is not intended.

We use the Weyl quantisation
$$
\big ( \Opw (\ab )u\big )(x) = (2\pi \hbar )^{-3}\int \limits _{\R ^6}
e^{i(x-y)\cdot \x /\hbar } \ab \Big ( \frac{x+y}{2},\x \Big ) u(y)\, dy \, d\x ,
$$
for which $\Dl _j = \Opw (\dsym _j)$ exactly, and, for $\asf ,\bsf \in \R $, the symbol
classes
$$
\Sigma ^{\asf ,\bsf } := \Big \{ \ab \in \cs (\R ^6;\Mfour ) \, : \,
|\partial _x^\a \partial _\x ^\g \ab (x,\x )| \le C_{\a \g }
\kum{x}^{-\asf -|\a |}\kum{\x }^{-\bsf -|\g |} \Big \} ,
$$
that is, H\"ormander's classes $S(\kum{x}^{-\asf }\kum{\x }^{-\bsf }, g_0)$ for the metric
$g_0 = \kum{x}^{-2}|dx|^2 + \kum{\x }^{-2}|d\x |^2$, whose Planck function is
$(\kum{x}\kum{\x })^{-1}$; each order of the semiclassical expansion therefore gains one
unit of decay in \emph{both} variables. Under the hypotheses of
Theorem~\ref{ZCMM2026d:weyl_trace_formula} we have $\Pot _j \in \Sigma ^{0,0}$,
$\Wpot \in \Sigma ^{\d , 0}$ and $\dsym _j - z \in \Sigma ^{0,-1}$ uniformly for $z$ in a
compact set. We shall use the following standard facts of the Weyl calculus, all of which
hold verbatim for $\Mfour $-valued symbols, the only change being that the Moyal expansion
keeps the order of the matrix factors; see \cite[\S 18.4--18.6]{hr_3},
\cite[Chapters 7 and 9]{dimassi_sjostrand99} and \cite[Chapter 2]{martinez02}.

\begin{itemize}
\item[(Q1)] \emph{Composition.} If $\ab \in \Sigma ^{\asf _1,\bsf _1}$ and
$\bb \in \Sigma ^{\asf _2,\bsf _2}$ then
$\Opw (\ab )\Opw (\bb ) = \Opw (\ab \# \bb )$ with
$\ab \# \bb \in \Sigma ^{\asf _1+\asf _2, \bsf _1 + \bsf _2}$ and, for every $N$,
$$
\ab \# \bb - \sum _{k<N}\frac{1}{k!}\Big ( \frac{i\hbar }{2}\Big )^k
\big ( \partial _\x \cdot \partial _y - \partial _x\cdot \partial _\y \big )^k
\ab (x,\x )\bb (y,\y )\Big |_{\substack{y=x \\ \y = \x }}
$$
$$
\in \ \hbar ^N \Sigma ^{\asf _1+\asf _2+N, \; \bsf _1+\bsf _2+N},
$$
with seminorms bounded by finitely many seminorms of $\ab $ and $\bb $. The term $k=0$ is
$\ab \bb $.
\item[(Q2)] \emph{Boundedness.} If $\ab \in \Sigma ^{0,0}$ then $\Opw (\ab )$ is bounded
on $L^2$, with norm bounded by $C\sum _{|\a |+|\g |\le 7}\sup |\partial ^\a _x
\partial _\x ^\g \ab |$, uniformly in $\hbar \in (0,1]$.
\item[(Q3)] \emph{Trace class and trace.} If $\partial ^\a \ab \in L^1(\R ^6)$ for
$|\a | \le 7$, then $\Opw (\ab )\in \Sone $ with
$\| \Opw (\ab )\| _{\Sone } \le C\hbar ^{-3}\sum _{|\a |\le 7}
\| \partial ^\a \ab \| _{L^1}$ and
\begin{equation}
\Tr \Opw (\ab ) = (2\pi \hbar )^{-3}\int \limits _{\R ^6}\tr \, \ab (x,\x )\, dx\, d\x .
\label{ZCMM2026d:trace_identity}
\end{equation}
In particular $\ab \in \Sigma ^{\asf ,\bsf }$ with $\asf >3$ and $\bsf >3$ is admissible.
\item[(Q4)] \emph{Continuity.} $\ab \mapsto \Opw (\ab )$ is continuous from
$\Sigma ^{\asf ,\bsf }$ to $\bopp{L^2}$ for $\asf ,\bsf \ge 0$ and to $\Sone $ for
$\asf ,\bsf >3$; in particular it commutes with absolutely convergent symbol-valued
integrals.
\end{itemize}

Recall the Helffer--Sj\"ostrand formula: for $g\in \ccs (\R )$ with
$\supp g \subset [-R,R]$ and every $M\in \N $ there is an almost analytic extension
$\tilde{g} \in \ccs (\C )$ with $\tilde{g}|_\R = g$,
$\supp \tilde{g} \subset \{ |\re z| \le R, \, |\im z| \le 1\}$ and
$|\bar \partial \tilde{g}(z)| \le C_M |\im z|^M$, such that for any self-adjoint operator
$A$,
\begin{equation}
g(A) = -\frac{1}{\pi }\int \limits _\C \bar \partial \tilde{g}(z)(A-z)^{-1}\, \dLeb (dz),
\label{ZCMM2026d:HS}
\end{equation}
the integral converging in operator norm; the same formula, with the same $\tilde{g}$,
holds for a Hermitian matrix $\dsym (x,\x )$ in place of $A$.

\begin{lemma}
Let $\eb _j (x,\x ,z) := (\dsym _j(x,\x )-z)^{-1}$ for $|\re z| \le R$ and
$0<|\im z| \le 1$. Then
$$
\big |\partial _x^\a \partial _\x ^\g \eb _j\big | \le C_{\a \g }
|\im z|^{-1-|\a |-|\g |}\kum{x}^{-|\a |}\kum{\x }^{-1-|\g |},
$$
i.e.\ $\eb _j (\cdot ,z) \in \Sigma ^{0,1}$ with seminorms
$\Oc (|\im z|^{-1-|\a |-|\g |})$. Moreover $z\mapsto \eb _j(x,\x ,z)$ is holomorphic off
$\sm \dsym _j (x,\x ) \subset \R $ and, for
$\kum{\x } \ge \L := C_0 (1 + \sup _j \| \Pot _j\| _\infty + R)$, it is holomorphic on a
neighbourhood of $\supp \tilde{g}$.
\label{ZCMM2026d:resolvent_symbol}
\end{lemma}

\begin{proof}
Since $\dsym _j$ is Hermitian,
$\| \eb _j \| = \dist (z, \sm \dsym _j)^{-1}\le |\im z|^{-1}$. The eigenvalues of
$\dsym _j = \sum _k \alb _k \x _k + \beb + \Pot _j$ satisfy
$|\l | \ge \kum{\x } - 1 - \| \Pot _j\| _\infty $, because the unperturbed ones are
$\pm \kum{\x }$ and a Hermitian perturbation moves them by at most
$\| \Pot _j\| _\infty $. Hence for $\kum{\x } \ge \L $ we get
$\| \eb _j\| \le 2\kum{\x }^{-1}$ and the last assertion. Combined with
$\| \eb _j\| \le |\im z|^{-1}$ and $|\im z| \le 1$ this gives the case
$\a = \g = 0$.  For derivatives we start from $\partial\eb_j=-\eb_j(\partial\dsym_j)\eb_j$ and 
iterate this identity. A mixed derivative $\partial_x^\a\partial_\x^\g\eb_j$ is a finite sum of 
products of the form
\[
 \eb_j(\partial^{\nu_1}\dsym_j)\eb_j\cdots
 (\partial^{\nu_q}\dsym_j)\eb_j,
 \qquad \nu_1+\cdots+\nu_q=(\a,\g),
\]
with $1\le q\le |\a|+|\g|$; terms with higher $x$-derivatives of $\Pot_j$ may contain
fewer than $1+|\a|+|\g|$ resolvent factors. Since $0<|\im z|\le1$, replacing the actual
power of $|\im z|^{-1}$ by the larger bound
$|\im z|^{-1-|\a|-|\g|}$ is harmless. Moreover every $x$-derivative falling on
$\dsym_j$ gains the corresponding factor $\kum{x}^{-1}$, while every $\x$-derivative is
carried by resolvent factors and gains one power of $\kum{\x}^{-1}$ in the large-$\x$
region. Combining this with the zeroth-order bounds proves the asserted
$\Sigma^{0,1}$ estimates.
\end{proof}

\begin{lemma}
Fix $N \in \N $. There are symbols $\bb _{j,k}(\cdot ,z), \bt _{j,k}(\cdot ,z)
\in \Sigma ^{k,1+k}$, $0\le k<N$, with $\bb _{j,0}=\bt _{j,0} = \eb _j$, whose seminorms
are $\Oc (|\im z|^{-M_k})$, which are holomorphic in $z$ wherever $\eb _j$ is, and such
that, with $\bb _j^{(N)} = \sum _{k<N}\hbar ^k \bb _{j,k}$ and
$\bt _j^{(N)}=\sum _{k<N}\hbar ^k \bt _{j,k}$,
$$
\bb _j^{(N)}\# (\dsym _j - z) = 1 + \hbar ^N \rb _{j,N}, \qquad
(\dsym _j - z) \# \bt _j^{(N)} = 1 + \hbar ^N \tilde{\rb }_{j,N},
$$
with $\rb _{j,N}, \tilde{\rb }_{j,N} \in \Sigma ^{N,N}$ with seminorms
$\Oc (|\im z|^{-M_N})$. Consequently, with $\Rb _j(z) = (\Dl _j - z)^{-1}$,
\begin{equation}
\Rb _j = \Opw (\bb _j^{(N)}) - \hbar ^N \Opw (\rb _{j,N})\Rb _j , \qquad
\Rb _j = \Opw (\bt _j^{(N)}) - \hbar ^N \Rb _j \Opw (\tilde{\rb }_{j,N}).
\label{ZCMM2026d:parametrices}
\end{equation}
\label{ZCMM2026d:parametrix_lemma}
\end{lemma}

\begin{proof}
Since $\dsym _j - z \in \Sigma ^{0,-1}$ and $\eb _j \in \Sigma ^{0,1}$, the standard
recursion
$$
\bb _{j,0}=\eb _j , \qquad
\bb _{j,k} = -\Big ( \sum _{l=1}^k \cb _l \big ( \bb _{j,k-l}, \dsym _j - z\big )\Big )
\eb _j \quad (k\ge 1),
$$
where $\cb _l$ denotes the $l$-th term in the Moyal expansion (Q1) with $\hbar ^l$
factored out, produces symbols with $\cb _l (\bb _{j,k-l},\dsym _j -z) \in \Sigma ^{k,k}$,
hence $\bb _{j,k}\in \Sigma ^{k,1+k}$; the bounds in $|\im z|$ follow inductively from
Lemma~\ref{ZCMM2026d:resolvent_symbol}. With this choice all terms of order $\hbar ^k$,
$1\le k<N$, cancel and the remainder lies in $\hbar ^N \Sigma ^{N,N}$ by the remainder
statement in (Q1). The left parametrix is constructed in the same way. Quantising and
multiplying by $\Rb _j$ yields \eqref{ZCMM2026d:parametrices}. Holomorphy in $z$ is clear,
since each $\bb _{j,k}$ is a finite sum of products of $\eb _j$'s and derivatives of
$\dsym _j$, hence rational in $z$ with poles only in $\sm \dsym _j(x,\x )\subset \R $.
\end{proof}

\begin{lemma}
Let $\cb (x,\x ,z)$ be a symbol which, for each $(x,\x )$ with $\kum{\x }\ge \L $, extends
holomorphically in $z$ to a neighbourhood of $\supp \tilde{g}$, with $|\cb |$ locally
bounded there. Then
$$
\s (x,\x ) := \frac{1}{\pi }\int \limits _\C \bar \partial \tilde{g}(z)\,
\cb (x,\x ,z)\, \dLeb (dz) = 0 \qquad \text{whenever }\kum{\x }\ge \L .
$$
\label{ZCMM2026d:stokes_lemma}
\end{lemma}

\begin{proof}
For such $(x,\x )$ we have $\bar \partial (\tilde{g}\cb ) = (\bar \partial \tilde{g})\cb $,
and $\tilde{g}\cb $ is compactly supported and $C^1$ in $z$, so that
$\int _\C \bar \partial (\tilde{g}\cb )\, \dLeb (dz) = 0$ by Stokes' theorem.
\end{proof}

Lemma~\ref{ZCMM2026d:stokes_lemma} is the mechanism of the proof: the Dirac resolvent is a
symbol of order $-1$ only, and no product of two of them is integrable in
$\x \in \R ^3$; but after integration against $\bar \partial \tilde{g}$ the resulting
symbol is \emph{compactly supported} in $\x $, and integrability is free. It is the
counterpart, at the level of the full symbol, of the elementary fact recorded in
Lemma~\ref{ZCMM2026d:symbol_eigenvalues} that $g(\dsym _j)$ vanishes for $|\x |>R$.

\begin{proof}[Proof of Theorem~\ref{ZCMM2026d:weyl_trace_formula}]
Fix $N=4$ and let $M$ be chosen at the end of Step 3.
\newline
\noindent
\emph{Step 1: the basic identity.} By the second resolvent identity
$\Rb _2 - \Rb _1 = -\Rb _2 \Wpot \Rb _1$, so that \eqref{ZCMM2026d:HS} gives
\begin{equation}
g(\Dl _2) - g(\Dl _1) = \frac{1}{\pi }\int \limits _\C \bar \partial \tilde{g}(z)\,
\Rb _2(z)\Wpot \Rb _1(z)\, \dLeb (dz),
\label{ZCMM2026d:basic_identity}
\end{equation}
the integral converging in operator norm, because
$$
\| \Rb _2 \Wpot \Rb _1 \| \le \| \Wpot \| _\infty |\im z|^{-2}.
$$
The same computation for matrices, together with
$\eb _2 - \eb _1 = -\eb _2 \Wpot \eb _1$, gives the classical counterpart
\begin{equation}
\frac{1}{\pi }\int \bar \partial \tilde{g}\, \eb _2 \Wpot \eb _1 \, \dLeb (dz)
= \frac{1}{\pi }\int \bar \partial \tilde{g}\, (\eb _1 - \eb _2)\, \dLeb (dz)
= g(\dsym _2) - g(\dsym _1).
\label{ZCMM2026d:classical_identity}
\end{equation}
\noindent
\emph{Step 2: splitting with the two parametrices.} 
By Lemma~\ref{ZCMM2026d:parametrix_lemma}, insert the right parametrix
$\Rb_2=\Opw(\bt_2^{(N)})-\hbar^N\Rb_2\Opw(\tilde{\rb}_{2,N})$ for $\Rb_2$, and the
\emph{left} parametrix
$\Rb_1=\Opw(\bb_1^{(N)})-\hbar^N\Opw(\rb_{1,N})\Rb_1$ for $\Rb_1$. 
By (Q1),
\begin{align}
\Rb _2 \Wpot \Rb _1 &= \Opw \big ( \bt _2^{(N)}\# \Wpot \# \bb _1^{(N)}\big )
- \hbar ^N \Opw (\bt _2^{(N)})\Opw \big ( \Wpot \# \rb _{1,N}\big )\Rb _1 \nonumber \\
&\phantom{=} - \hbar ^N \Rb _2 \Opw \big ( \tilde{\rb }_{2,N}\# \Wpot \big )
\Opw (\bb _1^{(N)}) \nonumber \\
&\phantom{=} + \hbar ^{2N}\Rb _2 \Opw \big ( \tilde{\rb }_{2,N}\# \Wpot \# \rb _{1,N}
\big )\Rb _1 .
\label{ZCMM2026d:four_terms}
\end{align}
The point of using one left and one right parametrix is that in the second term the
smoothing factor $\rb _{1,N}$ sits immediately to the \emph{right} of $\Wpot $, and in the
third the factor $\tilde{\rb }_{2,N}$ immediately to its \emph{left}; in both cases the
two compose, by (Q1), into a symbol in $\Sigma ^{\d +N,N}$, which is of trace class by
(Q3) since $\d >3$ and $N=4>3$. Had the same parametrix been used on both sides, one of
the remainder terms would have had the decaying factor $\Wpot $ separated from every
smoothing factor by a bare resolvent, and no trace-ideal estimate would have been
available.
\newline
\noindent
\emph{Step 3: the terms.} By (Q1),
$$
\bt _2^{(N)}\# \Wpot \# \bb _1^{(N)} = \cb _0(z) + \hbar \cb _1(z) + \hbar ^2 \varrho_2  (z),
$$
where $\cb _0 = \eb _2 \Wpot \eb _1 \in \Sigma ^{\d ,2}$,
$\cb _1 \in \Sigma ^{\d +1,3}$ and $\varrho_2 \in \Sigma ^{\d +2,4}$, all with seminorms
$\Oc (|\im z|^{-M'})$, and $\cb _0,\cb _1$ holomorphic in $z$ wherever $\eb _1,\eb _2$
are. Setting
$$
\s _k := \frac{1}{\pi }\int \bar \partial \tilde{g}\, \cb _k \, \dLeb (dz) \ (k=0,1),
\qquad \hat{\varrho}_{2} := \frac{1}{\pi }\int \bar \partial \tilde{g}\, \varrho_2 \, \dLeb (dz),
$$
which converge absolutely in every seminorm provided $M \ge M'+2$, we conclude from
Lemma~\ref{ZCMM2026d:stokes_lemma} that $\s _0 = \s _1 = 0$ for $\kum{\x }\ge \L $; thus
$\s _0,\s _1 \in \Sigma ^{\d ,\infty }$ are compactly supported in $\x $, while
$ \hat{\varrho}_{2} \in \Sigma ^{\d +2,4}$. All three are admissible for (Q3), because
$\int \kum{x}^{-\d }\1 _{\{ \kum{\x }\le \L \}}\, dx \, d\x < \infty $ (here $\d >3$ is
used) and $\int \kum{x}^{-\d -2}\kum{\x }^{-4}\, dx\, d\x < \infty $. Hence, by (Q4), the
first term of \eqref{ZCMM2026d:four_terms} contributes
$\Opw (\s _0) + \hbar \Opw (\s _1) + \hbar ^2 \Opw ( \hat{\varrho}_{2} )$, all of trace class with
trace norms $\Oc (C(g)\hbar ^{-3})$.

We estimate the remaining terms separately. By (Q1),
\[
 \Wpot\#\rb_{1,N},\quad \tilde{\rb}_{2,N}\#\Wpot
 \in \Sigma^{\d+N,N},
 \qquad
 \tilde{\rb}_{2,N}\#\Wpot\#\rb_{1,N}
 \in \Sigma^{\d+2N,2N},
\]
with seminorms bounded by a fixed power $|\im z|^{-M''}$. Since $N=4$ and $\d>3$, all
three symbols are trace-class symbols by (Q3). Thus
\[
 \|\Opw(\Wpot\#\rb_{1,N})\|_{\Sone}
 +\|\Opw(\tilde{\rb}_{2,N}\#\Wpot)\|_{\Sone}
 \le C\hbar^{-3}|\im z|^{-M''},
\]
and
\[
 \|\Opw(\tilde{\rb}_{2,N}\#\Wpot\#\rb_{1,N})\|_{\Sone}
 \le C\hbar^{-3}|\im z|^{-M''}.
\]
The other factors are bounded. More precisely,
\[
 \|\Rb_j(z)\|\le |\im z|^{-1},
\]
and, by (Q2), after increasing $M''$ if necessary,
\[
 \|\Opw(\bb_j^{(N)})\|+\|\Opw(\bt_j^{(N)})\|
 \le C|\im z|^{-M''}.
\]
Since $\Sone$ is a two-sided ideal, the second and third terms of
\eqref{ZCMM2026d:four_terms}, including their prefactor $\hbar^N$, have trace norm
$\Oc(\hbar^{N-3}|\im z|^{-K})$ for some fixed $K$. The fourth term, including its
prefactor $\hbar^{2N}$, has trace norm $\Oc(\hbar^{2N-3}|\im z|^{-K})$.
Choose the order of the almost-analytic extension so large that
\[
 \int_{\C}|\bar\partial\tilde g(z)|\,|\im z|^{-K}\,\dLeb(dz)<\infty.
\]
After integration, the three remainder terms therefore contribute
$\Oc(C(g)\hbar^{N-3})=\Oc(C(g)\hbar)$ in trace norm. Together with the first term this
proves $g(\Dl_2)-g(\Dl_1)\in\Sone$ and
\[
 \|g(\Dl_2)-g(\Dl_1)\|_{\Sone}\le C(g)\hbar^{-3}.
\]
\noindent
\emph{Step 4: taking the relative trace.} By \eqref{ZCMM2026d:trace_identity},
\[
\Tr\big(g(\Dl_2)-g(\Dl_1)\big)
=(2\pi\hbar)^{-3}\int_{\R^6}\tr\s_0(x,\x)\,dx\,d\x
+\Oc\big(C(g)\hbar^{-2}\big).
\]
Indeed, the term $\hbar\Opw(\s_1)$ contributes $\Oc(C(g)\hbar^{-2})$, the term
$\hbar^2\Opw( \hat{\varrho}_{2})$ contributes $\Oc(C(g)\hbar^{-1})$, and the integrated remainder
terms from Step~3 contribute only $\Oc(C(g)\hbar)$. By the definition of $\s_0$ and the
matrix identity \eqref{ZCMM2026d:classical_identity},
\[
 \s_0(x,\x)=g(\dsym_2(x,\x))-g(\dsym_1(x,\x))
\]
pointwise. Hence
\[
\Tr\big(g(\Dl_2)-g(\Dl_1)\big)
=(2\pi\hbar)^{-3}\int_{\R^6}
\tr\big(g(\dsym_2)-g(\dsym_1)\big)\,dx\,d\x
+\Oc\big(C(g)\hbar^{-2}\big).
\]
This is the asserted relative trace formula. The constant $C(g)$ depends only on finitely
many seminorms of $g$ (through the chosen almost-analytic extension) and on a bound for
$\supp g$. In particular, no separate quantities $\Tr g(\Dl_j)$ are used or required. 
\end{proof}

\section{From analytic singularities to resonances}
\label{ZCMM2026d:resonances_near_analytic_sing}
The analysis rests on two distributions on the real line: one attached to the pair of mass functions, one attached to the pair of symbols. For $\f \in \ccs (\R )$ we set
\begin{equation} 
\langle \m , \f \rangle := \int \limits _{\R ^3}\big ( \f (m_2(x)) - \f (m_1(x))\big )\, dx
\label{ZCMM2026d:mu_def}
\end{equation} 
and
\begin{equation} 
\langle \o , \f \rangle := \int \limits _{\R ^6}\Big ( \f (p_{+,2}(x,\x ))
- \f (p_{+,1}(x,\x )) + \f (p_{-,2}(x,\x )) - \f (p_{-,1}(x,\x ))\Big )\, dx \, d\x .
\label{ZCMM2026d:omega_def}
\end{equation} 
Thus $\m$ compares the two mass functions and $\o$ compares the two
symbols.  Both are relative objects, and, for a general test 
function $\f$, neither
\[
   \int_{\R^3}\f(m_j(x))\,dx
\]
nor its phase-space counterpart need converge separately. The distribution $\m$ is compactly supported and is of order at most one. Indeed,
\[
   \supp(\m)\subset    \left[      \min_{j=1,2}\inf_{x\in\R^3}m_j(x),  \max_{j=1,2}\sup_{x\in\R^3}m_j(x)    \right]  \subset(0,\infty),
\]
and, by the mean-value theorem,
\[
   |\langle \m,\f\rangle|    \leq    \|\f'\|_{\infty}    \int_{\R^3}|m_2(x)-m_1(x)|\,dx    =
   \|\f'\|_{\infty}    \|v_2-v_1\|_{L^1(\R^3)}.
\]
Here $v_2-v_1\in L^1(\R^3)$ by Assumption~$(\Ab _\d )$, since $\delta>3$. The distribution $\o$ is likewise of order at most one, but it is not compactly supported.  More precisely,
\[
   \supp(\o)\subset (-\infty,-c_0]\cup[c_0,\infty),
\]
because $p_{+,j}\geq c_0$ and $p_{-,j}\leq-c_0$.  To see the order-one bound, let $\supp\f\subset[-R,R]$.  Then $\f(p_{\pm,j}(x,\x))$ can be non-zero only when $\tau_j(x,\x)\leq R$, and hence only when $|\x|\leq R$.  Moreover, since
\[
 |\tau_2(x,\x)-\tau_1(x,\x)|    \leq |m_2(x)-m_1(x)|    =|v_2(x)-v_1(x)|,
\]
the mean-value theorem gives
\[
|\langle\o,\f\rangle|   \leq    2\, \mbox{vol}\, B(0,R)\,  \|\f'\|_{\infty}    \|v_2-v_1\|_{L^1(\R^3)}.
\]
Thus $\o\in \ddist(\R)$ locally with order at most one, while its support
may extend to arbitrarily large positive and negative energies.

On the half-line $E>1$ the distribution $\m $ can be read off from the distribution
functions of the two mass functions. Put
$$
\n _{+,j}(E) := \int \limits _{m_j(x)\ge E}dx , \qquad E>1 ,
$$
which is finite because $m_j \to 1$ at infinity, is decreasing, 
and $\n _{+,j}(E)>0$ exactly for $1<E<\sup m_j$ and 
$\n _{+,j}(E)=0$ for $E\ge \sup m_j$;  let $\m _{+,j} := -\, d\n _{+,j}/dE$ be the associated positive measure,
of locally finite mass on $(1,\infty )$, with
$\supp \m _{+,j} \subset [1, \sup m_j]$. For $\f \in \ccs ((1,\infty ))$ an integration by
parts gives $\langle \m _{+,j},\f \rangle = \int \f '(E)\n _{+,j}(E)\, dE = \int \f (m_j(x))\, dx$, so that
\begin{equation} 
\m |_{(1,\infty )} = \m _{+,2} - \m _{+,1} .
\label{ZCMM2026d:mu_restricted}
\end{equation} 
Only this half-line will be needed: by Corollary~\ref{ZCMM2026d:omega_even} below the
distribution $\o $ is even, and the analysis near $-E_0$ reduces to the analysis near
$E_0$.

\subsection{Main results}
Here we present and discuss the main results, i.e.\ how certain analytic singularities of
$\m $ defined in \eqref{ZCMM2026d:mu_def} generate resonances of the Dirac operator. We phrase
our main theorem for $E_0>1$.

\begin{theorem}
Suppose $\Dl _j = \Dl _0 + \Pot _j$ where $\Pot _j = \beb v_j$, $j=1,2$, satisfies
Assumption $(\Ab _\d )$.  Let $E_0 > 1$ be a boundary point of
$\supp (\m )$. Then, for any complex neighborhood $U$ of $E_0$ there exists a positive
constant $C = C(U)$ such that
$$
\sum _{j=1}^2 \# \big ( \Res (\Dl _j)\cap U\big ) \ge C \hbar ^{-3},
$$
provided $\hbar $ is small enough. The same holds near $-E_0$.
\label{ZCMM2026d:thm_rel}
\end{theorem}

Given $v_2$ it is possible to construct $v_1$ so that Assumption $(\Ab _\d )$ holds but
$\Dl _1$ has no resonances near $E_0$. We may then invoke the previous result to obtain
the following result for a single Dirac operator.

\begin{theorem}
Let $v_2\in C^\infty(\R^3;\R)$ be bounded and suppose that, for some $\ve\in(0,1/2)$, 
$R_0\ge0$ and $\r>0$, it extends holomorphically to 
$C_{\ve,R_0}$ with $|v_2(z)|=\Oc(\kum{z}^{-\r})$. 
Assume, moreover, that $\inf_{x\in\R^3}(1+v_2(x))>0$ 
and put $E_0:=1+\sup_x v_2(x)>1$. Then, for any complex neighbourhood $U$ of $E_0$ there exists a constant $C=C(U)$ such that
$$
\# \big ( \Res (\Dl _2) \cap U \big ) \ge C\hbar ^{-3},
$$
provided $\hbar $ is small enough, and likewise near $-E_0$.
\label{ZCMM2026d:thm_abs}
\end{theorem}

In physical units the two energies read $\pm (mc^2 + \sup V)$, and the hypothesis
$E_0>1$ reads $\sup V >0$: only a potential which \emph{increases} the local mass produces
resonances by this mechanism. If $\sup v_2 \le 0$, then $\max \supp (\m ) \le 1$ and the
above theorems do not apply; the corresponding analytic singularity of $\o $ then lies in
the spectral gap $(-1,1)$, where Assumption $(\Ab _\O ^\pm )$ fails and where the relevant
statement concerns the \emph{eigenvalues} of $\Dl _2$ rather than its resonances. Indeed,
for $0<E<1$ the classical set $\{ p_+ \le E\} = \{ m(x)^2 + |\x |^2 \le E^2 \}$ is
compact.

\begin{example}
Let $v_2(x) = c_{\ast}  \kum{x}^{-4}$ with $c_{\ast}   >0$, so that $m_2 = 1 + c_{\ast}   \kum{x}^{-4}$ and $E_0 = 1+c_{\ast}   $. Note that $v_2$ satisfies the hypotheses of Theorem~\ref{ZCMM2026d:thm_abs}. Indeed, $(1+z\cdot z)^{-2}$ is 
holomorphic and $\Oc (\kum{z}^{-4})$ on $C_{\ve ,R_0}$ because, there,  
$\re (1+z\cdot z)\geq 1+(1-\ve ^2)|\re z|^2>0$.  For $E>1$ the condition 
$m_2(x)\geq E$ reads $|x|^2 \leq ( c_{\ast}   /(E-1))^{1/2}-1$, whence
$$
\n _{+,2}(E) = \frac{4\pi }{3}\Big [ \Big ( \frac{c_{\ast}   }{E-1}\Big )^{1/2} - 1 \Big ]_+^{3/2},
$$
so that $\supp \n _{+,2} = [1, E_0]$. Writing $E = E_0 - \eta $ we get 
$(c_{\ast}   /(E-1))^{1/2} = (1-\eta / c_{\ast}   )^{-1/2} = 1 + \eta /(2 c_{\ast}   ) + \Oc (\eta ^2)$, hence
$$
\n _{+,2}(E) = \frac{4\pi }{3}\Big ( \frac{(E_0-E)_+}{2 c_{\ast}   }\Big )^{3/2}
\big ( 1 + \Oc (E_0-E)\big ) ,
$$
so that $\m _{+,2} = -\, d\n _{+,2}/dE$ is absolutely continuous near $E_0$, with density
$$
\m _{+,2}(E) = \frac{\pi }{ c_{\ast}   }\Big ( \frac{(E_0-E)_+}{2 c_{\ast}   }\Big )^{1/2}
\big ( 1 + \Oc (E_0-E)\big ) .
$$
Thus $\m $, which coincides with $\m _{+,2}$ near $E_0$ once the comparison potential of
Appendix~\ref{ZCMM2026d:appendix_comparison} has been chosen, has a square-root singularity
at the end point $E_0$ and is in particular not real analytic there. Of course no such
computation is needed in general: Lemma~\ref{ZCMM2026d:endpoint_lemma} produces the analytic
singularity from the support property alone.
\end{example}

\begin{remark}
The two families of resonances produced by Theorem~\ref{ZCMM2026d:thm_abs} are mirror images
of one another, by Lemma~\ref{ZCMM2026d:susy_lemma}; the family near $E_0$ lies below the real
axis and the one near $-E_0$ above it, each on the sheet uncovered by the corresponding
branch of $\G _{\th _0}$. This is in contrast with the electric case treated in
\cite{jkmm14}, where the two families sit near the unrelated points $\sup v +1$ and
$\inf v -1$.
\end{remark}

\begin{remark}
The constant $C(U)$ produced by the proof is exponentially small in the FBI parameter
$\l _j$ fixed in Step 6 below, and it degenerates as $U$ shrinks. Thus the \emph{power}
$\hbar ^{-3}$ is optimal, since it matches the upper bound \eqref{ZCMM2026d:upper_bound}, but
the constant carries no quantitative information; in particular the argument does not
produce a Weyl-type lower bound with an identifiable constant.
\label{ZCMM2026d:remark_constant}
\end{remark}

\subsection{Proofs of main results}
The proof of the existence of resonances relies on the fact that we can find points in
$\wfa (\o )$. Since the point $(E_0,\pm 1)$ in Theorem~\ref{ZCMM2026d:thm_rel} belongs to
$\wfa (\m )$ by Lemma~\ref{ZCMM2026d:endpoint_lemma}, we begin by relating $\o $ to $\m $. As
explained in the introduction, the relation is not a convolution; it becomes one after the
substitution $u=E^2$, and the resulting convolution operator is elliptic classical
analytic, which enables us to apply Proposition~\ref{ZCMM2026d:elliptic_wavefront}. Finally,
to prove Theorem~\ref{ZCMM2026d:thm_abs} we construct, 
given $v_2$, a ``small'' potential $v_1$ such that Assumption $(\Ab _\d )$ is fulfilled and $\Dl _1$ has no resonances near
$E_0$; this is carried out in Appendix~\ref{ZCMM2026d:appendix_comparison}.

\begin{lemma}
For every $\f \in \ccs ((0,\infty ))$ we have
\begin{equation}
\langle \o , \f \rangle = \langle \m , \F _\f \rangle , \qquad
\F _\f (t) = 4\pi \int \limits _{\R } \f (s) \, s \, (s^2-t^2)_+^{1/2}\, ds ,
\label{ZCMM2026d:kernel_identity}
\end{equation}
and $\F _\f \in C^1(\R )$. In other words $\o = K \m $ on $(0,\infty )$, where $K $
is the integral operator with kernel $K (E,t) = 4\pi E (E^2-t^2)_+^{1/2}$.
\label{ZCMM2026d:kernel_lemma}
\end{lemma}

\begin{proof}
Since $p_{-,j} = -\t _j \le -c_0 <0$, the two terms of \eqref{ZCMM2026d:omega_def} containing
$p_{-,j}$ vanish for $\f \in \ccs ((0,\infty ))$, whence
$$
\langle \o ,\f \rangle = \int \limits _{\R ^3}
\big [ \F _\f (m_2(x)) - \F _\f (m_1(x)) \big ] \, dx ,
\qquad \F _\f (t) = \int \limits _{\R ^3}\f \big ( (t^2+|\x |^2)^{1/2}\big )\, d\x .
$$
Passing to polar coordinates in the $\x $-variable and substituting
$s = (t^2+r^2)^{1/2}$, so that $r^2 \, dr = s(s^2-t^2)^{1/2}\, ds$, we obtain
$$
\F _\f (t) = 4\pi \int \limits _0^\infty \f \big ( (t^2+r^2)^{1/2}\big ) r^2 \, dr
= 4\pi \int \limits _{|t|}^\infty \f (s)\, s\, (s^2-t^2)^{1/2}\, ds ,
$$
which is \eqref{ZCMM2026d:kernel_identity}. Differentiating under the integral sign,
$$
\F _\f '(t) = -4\pi t \int \f (s) \, s\, (s^2-t^2)_+^{-1/2}\, ds ,
$$
an absolutely convergent integral depending continuously on $t$; hence
$\F _\f \in C^1(\R )$ and the pairing with the distribution $\m $, which is of order
$\le 1$, is well defined. The interchange of the order of integration is justified by
$|\F _\f (m_2) - \F _\f (m_1)| \le \| \F _\f ' \| _\infty |v_2 - v_1| \in L^1(\R ^3)$.
\end{proof}

The kernel $K $ depends on $(E,t)$ through $E^2-t^2$ and not through $E-t$; its
singular support is the pair of diagonals $\{ E = \pm t\}$. The next lemma removes this
difficulty.

\begin{lemma}
Let $q:(0,\infty ) \to (0,\infty )$, $q(t)=t^2$, and let $q_*$ denote the push forward as
in \eqref{ZCMM2026d:wf_invariance}. Then
\begin{equation}
q_* \o = 2\pi \big ( u_+^{1/2}\big ) \ast \, q_*\m \qquad \text{on }(0,\infty ),
\label{ZCMM2026d:structure}
\end{equation}
and the convolution operator $A := 2\pi \, u_+^{1/2}\ast \, \cdot $ satisfies
$\wfa (Au) = \wfa (u)$ for every $u\in \edist ((0,\infty ))$, the space of compactly supported distributions on $(0,\infty)$. 
Consequently
\begin{equation}
\wfa (\o ) \cap \big ( (0,\infty )\times (\R \setminus \{ 0\})\big ) =
\wfa (\m ) \cap \big ( (0,\infty )\times (\R \setminus \{ 0\})\big ) .
\label{ZCMM2026d:wf_transfer}
\end{equation}
\label{ZCMM2026d:structure_lemma}
\end{lemma}

\begin{proof}
\emph{The identity \eqref{ZCMM2026d:structure}.} Note first that, by Assumption
$(\Ab _\d )$, $\supp \m $ is a compact subset of $(0,\infty )$, so that
$q_*\m \in \edist ((0,\infty ))$ and the push forward is well defined. Let
$\f \in \ccs ((0,\infty ))$ and put $\p := \f \circ \sqrt{\cdot } \in \ccs ((0,\infty ))$.
Substituting $u=s^2$ in \eqref{ZCMM2026d:kernel_identity},
$$
\F _\f (t) = 4\pi \int \f (s) s(s^2-t^2)_+^{1/2}\, ds
= 2\pi \int \p (u) (u-t^2)_+^{1/2}\, du = \Xi _\p (t^2),
$$
where
$$
\Xi _\p (\s ) := 2\pi \int \limits _0^\infty \p (\s + w) w^{1/2}\, dw
$$
is smooth, with $\supp \Xi _\p \subset (-\infty , \max \supp \p ]$. Hence, by
Lemma~\ref{ZCMM2026d:kernel_lemma} and the definition of $q_*$,
\begin{align*}
\langle \o , \f \rangle &= \langle \m , \Xi _\p \circ q \rangle
= \langle q_*\m , \Xi _\p \rangle \\
&= \Big \langle q_*\m (\s ) , 2\pi \int \p (u)(u-\s )_+^{1/2}\, du \Big \rangle
= \big \langle 2\pi u_+^{1/2}\ast q_*\m , \, \p \big \rangle .
\end{align*}
Since $\f = \p \circ q$, the left-hand side equals $\langle q_*\o , \p \rangle $, which
proves \eqref{ZCMM2026d:structure}.

\emph{The operator $A$.} With the one-dimensional unitary Fourier transform
\[
(\Fr f)(\x )=(2\pi)^{-1/2}\int_{\R} e^{-it\x }f(t)\,dt,
\]
the convolution theorem gives $A=\Fr^{-1}a_A\Fr$, where, as homogeneous distributions,
\[
a_A(\x ):=2\pi\sqrt{2\pi}\,\Fr[u_+^{1/2}](\x )
=2\pi\G(\tfrac32)(i\x+0)^{-3/2}, 
\]
which is positively homogeneous of degree $-3/2$, real analytic and non-vanishing on
$\R \setminus \{ 0\}$. Fix $B>0$ and $\h _B \in \ccs (\R )$ equal to 1 on $[-B,B]$, and
split $a_A = \h _B a_A + (1-\h _B)a_A$. On $|\x | \ge B$ we have
$a_A(\x ) = c_\pm |\x |^{-3/2}$ with constants $c_\pm \ne 0$ on $\pm \x >0$, so that
\[
|\partial_{\x}^{N} a_A(\x )| = \Big ( \prod _{k=0}^{N-1}\big ( \tfrac32 + k\big )\Big )
|c_\pm | \, |\x |^{-3/2-N} \leq C^{N+1}N! \, |\x |^{-3/2-N},
\]
whence $(1-\h _B)a_A \in \Ssf_{\rm a}^{-3/2}(\R )$, and $|a_A(\x )| = 2\pi\G (\frac32)|\x |^{-3/2}$ shows that this high-frequency symbol is elliptic
on both components of $T^*\R \setminus 0$. Thus the corresponding high-frequency
operator is an elliptic analytic pseudodifferential operator of order $-3/2$ in the sense
of Proposition~\ref{ZCMM2026d:elliptic_wavefront}; changing its multiplier on a bounded
$\x $-set, if desired in order to place it in a standard global quantisation, changes the
operator only by a low-frequency term of the type considered next. On the other hand
$\h _B a_A$ is a compactly supported distribution, so its inverse Fourier transform is an
entire function of exponential type, and $\opc {\h _B a_A}u$  is real analytic for
$u\in \edist $. Since adding an analytic-regularising term does not affect the analytic
wave front set, Proposition~\ref{ZCMM2026d:elliptic_wavefront} gives
$\wfa (Au)=\wfa (u)$.

\emph{Conclusion.} Apply the preceding to $u = q_*\m $ and use \eqref{ZCMM2026d:wf_invariance}
on both sides of \eqref{ZCMM2026d:structure}; since $q$ is a real-analytic diffeomorphism of
$(0,\infty )$ and the covector transformation $\y \mapsto \y /2t$ is a positive scaling,
\eqref{ZCMM2026d:wf_transfer} follows.
\end{proof}

The proof of Lemma~\ref{ZCMM2026d:structure_lemma} shows that $\F _\f (t) = \Xi _\p (t^2)$ with $\Xi _\p $ smooth, 
so $\F _\f $ is in fact $C^\infty $, not merely $C^1$.

\begin{corollary}
Let $E_0 > 0$ be a boundary point of $\supp (\m )$. Then $(E_0,\pm 1)\in \wfa (\m )$ and
$(E_0,\pm 1)\in \wfa (\o )$.
\label{ZCMM2026d:wf_corollary}
\end{corollary}

\begin{proof}
The distribution $\m $ is real valued and $E_0$ belongs to $\supp (\m )$ without being an
interior point of it, so Lemma~\ref{ZCMM2026d:endpoint_lemma} gives
$(E_0,\pm 1)\in \wfa (\m )$. Since $E_0>0$, Lemma~\ref{ZCMM2026d:structure_lemma} applies.
\end{proof}

The corresponding statement on the negative half-line requires no separate argument,
because $\o $ is an even distribution.

\begin{corollary}
Let $\Rc \o $ denote the reflection of $\o $, defined by
$\langle \Rc \o , \f \rangle := \langle \o , \f (-\, \cdot \, )\rangle $.
Then $\Rc \o = \o$ in $\ddist(\R)$, that is,  
$\langle \o,\f(-\,\cdot\,)\rangle = \langle \o,\f\rangle$ for every 
$\f\in C^\infty_0(\R)$, and consequently
$$
(s _1E_0,s _2)\in \wfa (\o ) \mbox{ for all } s _1,s _2\in \{ \pm 1\}. 
$$
\label{ZCMM2026d:omega_even}
\end{corollary}

\begin{proof}
Let $\f \in \ccs ((-\infty ,0))$ and put $\check{\f }(s) := \f (-s) \in \ccs
((0,\infty ))$. Since $p_{\pm ,j}= \pm \t _j$, only the terms containing $p_{-,j}$ survive
in \eqref{ZCMM2026d:omega_def}, and
$$
\langle \o , \f \rangle = \int \limits _{\R ^3}\big [ \check{\F }_\f (m_2(x))
- \check{\F }_\f (m_1(x))\big ]\, dx ,
$$
$$
\check{\F }_\f (t) = \int \limits _{\R ^3}\f \big ( -(t^2+|\x |^2)^{1/2}\big )\, d\x .
$$
But $\check{\F }_\f = \F _{\check{\f }}$ with $\F $ as in
Lemma~\ref{ZCMM2026d:kernel_lemma}, and therefore
$$
\langle \o , \f \rangle = \langle \m , \F _{\check{\f }}\rangle
= \langle \o , \check{\f }\rangle
$$
by that lemma. Thus $\langle \Rc \o , \check{\f }\rangle
= \langle \o , \check{\f }\rangle $ for every $\check{\f } \in \ccs ((0,\infty ))$, and
the same identity on $(-\infty ,0)$ follows by exchanging the roles of the two
half-lines. Finally $\o $ vanishes near the origin, since $\f (\pm \t _j)=0$ whenever
$\supp \f \subset (-c_0,c_0)$; hence $\Rc \o = \o $ on all of $\R $. The assertion about
$\wfa (\o )$ now follows from Corollary~\ref{ZCMM2026d:wf_corollary}, the reflection
$E\mapsto -E$ being a real-analytic diffeomorphism.
\end{proof}

\begin{remark}
Corollary~\ref{ZCMM2026d:omega_even} makes the analytic singularity at $-E_0$ independent of
the spectral symmetry of Lemma~\ref{ZCMM2026d:susy_lemma}. The passage from that singularity to the resonance count near $-E_0$ may then be carried out either by repeating Steps 2--7
below with the boxes reflected in the origin 
using the mirror form of Theorem~\ref{ZCMM2026d:local_trace_formula}, in which the supremum in the error term is taken
over $\im z\ge 0$, since near $-E_0$ the resonances lie in the upper half plane. 
\label{ZCMM2026d:remark_negative}
\end{remark}

\begin{proof}[Proof of Theorem~\ref{ZCMM2026d:thm_rel}]
The parameters are fixed in an order which is not the order of the steps below, and we
record it here once. First a non-real distortion parameter $\th _0 = i\t \in D_\ve ^+$ with
$0<\t <\ve /(2\sqrt{1+\ve ^2})$: by distortion-independence this fixes the notion of
resonance, and with it the curve $\G _{\th _0}$ which the boxes have to avoid. Next, given
the complex neighbourhood $U$ of the boundary point $E_0>1$ of $\supp (\m )$, the box
parameters $b$ and then $a$ of Step~3, which depend on $U$, on $E_0$ and on
$\dist (E_0,\G _{\th _0})$ alone, and which determine $\O $, $W$ and the constant $C_0$ of
\eqref{ZCMM2026d:f_small}. Only then the cut-off $\h $, since
Theorem~\ref{ZCMM2026d:local_trace_formula} attaches it to $I=\O \cap \R $; and only then the FBI
data $\a _j,\b _j,\l _j,\kappa _j$, since these are read off from $\h \o $. Next $j$, taken
large in Step~6. Finally $\hbar $, which is allowed to tend to zero only once the wave packet
$f_j$ has been fixed, as Remark~\ref{ZCMM2026d:remark_gdep} requires. We nevertheless present the
wave packets first, in Step~2, because it is their size \eqref{ZCMM2026d:wave_packet_size} which
dictates the shape of the boxes; no step uses a quantity chosen at a later one.
\newline
\noindent
\emph{Step 1: the analytic singularity.}  By Corollary~\ref{ZCMM2026d:wf_corollary} we have $(E_0,\pm 1) \in \wfa (\o )$; we  concentrate on the covector $(E_0,+1)$.
\newline
\noindent
\emph{Step 2: FBI wave packets.} For $n=1$ the transform \eqref{ZCMM2026d:FBI} of a
distribution $u$ at a point $(\a ,\b ) \in \R \times \R $ reads
$$
(T_\l u)(\a ,\b ) = 2^{-1/2}\Big ( \frac{\l }{\pi }\Big )^{3/4}
\Big \langle u , \, \h \, e^{i\l \b (\a - \, \cdot \, ) - \l (\a - \, \cdot \, )^2/2}
\Big \rangle ,
$$
so that, up to the factor $2^{-1/2}(\l /\pi )^{3/4}$, testing $u$ against a Gaussian wave
packet \emph{is} taking its FBI transform. 
Fix a cut-off $\h \in \ccs ((1,\infty ))$ equal to one near $E_0$, which is possible since
$E_0>1$; it will be identified in Step 4 with the cut-off of
Theorem~\ref{ZCMM2026d:local_trace_formula}. By Step 1 the covector $(E_0,+1)$ belongs to
$\wfa (\o )$, that is, the FBI transform of $\h \o $ is \emph{not} exponentially small
there. By the definition of the analytic wavefront set, we obtain real sequences
$$
\a _j \to E_0 , \qquad \b _j \to 1 , \qquad \l _j \to +\infty , \qquad \kappa _j \to 0^+ ,
$$
such that the entire functions
\begin{equation}
f_j(E) := e^{i\l _j \b _j (\a _j - E) - \l _j (\a _j - E)^2/2} , \qquad E \in \C ,
\label{ZCMM2026d:wave_packet}
\end{equation}
satisfy $\big | (T_{\l _j}(\h \o ))(\a _j , \b _j)\big | \ge e^{-\kappa _j \l _j}$, or
equivalently
$$
\big | \langle \o , \h f_j \rangle \big | \ \ge \ c \, \l _j^{-3/4} e^{-\kappa _j \l _j} ,
\qquad c = \sqrt{2}\, \pi ^{3/4} .
$$
The polynomial prefactor can be absorbed into the exponential: putting
$$
\widetilde{\kappa }_j := \kappa _j + \frac{1}{\l _j}\log \frac{\l _j^{3/4}}{c}
\ \longrightarrow \ 0 \qquad (j \to \infty ),
$$
which tends to zero because $\l _j^{-1}\log \l _j$ does, we get
\begin{equation}
\big | \langle \o , \h f_j \rangle \big | \ \ge \ e^{-\widetilde{\kappa }_j \l _j} .
\label{ZCMM2026d:FBI_lower}
\end{equation}
From here on we write $\kappa _j$ for $\widetilde{\kappa }_j$; the only property used later is
$\kappa _j \to 0$. Finally, writing $E = x+iy$, the packet has modulus
\begin{equation}
|f_j(x+iy)| = \exp \Big ( \l _j \Big [ \b _j y
+ \frac{y^2 - (x-\a _j)^2}{2}\Big ]\Big ) ,
\label{ZCMM2026d:wave_packet_size}
\end{equation}
which is what governs the choice of the boxes in Step 3.
\newline
\noindent
\emph{Step 3: the two boxes.} Choose $b>0$ so small that
$$
[E_0 - 3b, E_0+3b] + i[-3b,3b] \subset U , \qquad b < \tfrac{E_0-1}{4},
$$
and then $a \in (0,b^2)$ so small that
\begin{equation}
2a + 2a^2 \le \frac{b^2}{16}, \qquad a \le \tfrac18 , \qquad
\big ( [E_0-3b,E_0+3b] + i[-3a,3a]\big ) \cap \G _{\th _0} = \emptyset ,
\label{ZCMM2026d:box_condition}
\end{equation}
the last condition being possible because $\dist (E_0, \G _{\th _0}) > 0$. Put
$$
\O=(E_0-2b,E_0+2b)+i(-2a,a), \qquad W = [E_0-b,E_0+b] + i(-a,a] .
$$
Then $\O $ satisfies $(\Ab _\O ^\pm )$: it is contained in $\{ |\re z| >1 \}$ because
$b<(E_0-1)/4$, it meets both half planes, and
$\overline{\O }\cap \G _{\th _0}=\emptyset $ by \eqref{ZCMM2026d:box_condition}; moreover
$W \subset U$.

We claim that, for $j$ so large that $|\a _j-E_0|\le b/2$ and $\frac12 \le \b _j\le 2$,
\begin{multline}
|f_j(E)| \le e^{-C_0\l _j} \quad \text{for }
E \in (\O \setminus W)\cup \{ z \, : \, \dist (z,\pd \O )<2\e_{\star} , \, \im z \le 0 \}, \\
\text{where } \quad C_0 := \min \Big \{ \frac a4 , \frac{b^2}{16}\Big \},
\label{ZCMM2026d:f_small}
\end{multline}
provided $\e_{\star} $ is small enough, and that
\begin{equation}
|f_j(E)| \leq M_j := e^{\l _j a^2/2} \qquad \text{for } E\in W, \, \im E \le 0 .
\label{ZCMM2026d:f_big}
\end{equation}
Indeed, let $E=x+iy$. If $|x-E_0|\ge b$, then $|x-\a _j|\ge b/2$.
On $\O $ we have $y<a$ and $|y|<2a$, and on the part of the $2\e_{\star} $-neighbourhood of $\pd \O $ under consideration we have $y\le 0$; in either case $\b _jy\le 2a$, because
$\b _j\le 2$ when $y>0$ and $\b _jy\le 0$ when $y\le 0$, and $y^2/2\le 2a^2$. Hence the
exponent in \eqref{ZCMM2026d:wave_packet_size} is at most
$\l _j[2a+2a^2-\frac{b^2}{8}]\le -\l _j\frac{b^2}{16}$ by \eqref{ZCMM2026d:box_condition}.
If $|x-E_0|<b$ and $y \le -a$, then $\b _jy \le -a/2$ and
$y^2/2 \le 2a^2$, so the exponent is at most
$\l _j[-\frac a2 + 2a^2] \le -\l _j \frac a4$ for $a \le \frac18$. 
Since a point of $\O$ with $|x-E_0|<b$ and $y>-a$ belongs to $W$,
the set $\O\setminus W$ is covered by these two cases.  The same is
true for $\pd\O\cap\{\im z\le0\}$, where either
$|x-E_0|=2b$ or $y=-2a$.  By continuity of the exponent in
\eqref{ZCMM2026d:wave_packet_size}, the same bound 
remains valid in a sufficiently small $2\e_\star$-neighbourhood of
$\pd\O\cap\{\im z\le0\}$, after decreasing $C_0>0$ if necessary.
This proves \eqref{ZCMM2026d:f_small}.
 Note that the upper edge $\{ y = a \}$ of $\pd \O $, on which
$|f_j|$ is exponentially \emph{large}, is excluded by the restriction $\im z \le 0$ in
Theorem~\ref{ZCMM2026d:local_trace_formula}. Finally, for $E \in W$ with $y\le 0$ we have
$\b _j y \le 0$, so that the exponent is at most $\l _j y^2/2 \le \l _j a^2/2$, which is
\eqref{ZCMM2026d:f_big}.

\emph{Step 4: the local trace formula.} Let $\h $ be the cut-off of
Theorem~\ref{ZCMM2026d:local_trace_formula} associated with $I = \O \cap \R $; for $b$ small
it is supported in $(1,\infty )$ and equals 1 near $E_0$, as required in Step 2. Applying
Theorem~\ref{ZCMM2026d:local_trace_formula} with $f = f_j$, and using \eqref{ZCMM2026d:f_small} to
bound the error term, we obtain
\begin{multline}
\Tr \big [(\h f_{j} )(\Dl _2) - (\h f_{j})(\Dl _1)\big ]   \\ 
= \sum _{z_k \in \Res (\Dl _2)\cap \O } f_j(z_k)
- \sum _{z_k \in \Res (\Dl _1)\cap \O } f_j(z_k) + e^{-C_0\l _j}\Oc (\hbar ^{-3}).
\label{ZCMM2026d:step_local}
\end{multline}

\emph{Step 5: from $\O $ to $W$.} Near $E_0$ the resonances lie in
$S_{\th _0}\cup \R \subset \{ \im z \le 0 \}$, so that
$\Res (\Dl _i)\cap \O \subset \O \cap \{ \im z \le 0\}$, where \eqref{ZCMM2026d:f_small}
applies off $W$. By the upper bound \eqref{ZCMM2026d:upper_bound},
$$
\Big | \sum _{z_k \in \Res (\Dl _i) \cap (\O \setminus W)} f_j(z_k)\Big |
\le e^{-C_0\l _j}\, \# \big ( \Res (\Dl _i)\cap \O \big ) \le C e^{-C_0\l _j}\hbar ^{-3},
$$
so the sums in \eqref{ZCMM2026d:step_local} may be replaced by sums over $W$, at the cost of
enlarging the constant in the error term.

\noindent
\textit{Step 6: the Weyl formula and the choice of the remaining parameters.} 
The mass-type potentials $\Pot _j = \beb v_j$ of \eqref{ZCMM2026d:Vmasstype} satisfy the hypotheses in 
Theorem~\ref{ZCMM2026d:weyl_trace_formula}  by Lemma~\ref{ZCMM2026d:cauchy_estimates}. By Lemma~\ref{ZCMM2026d:symbol_eigenvalues} and \eqref{ZCMM2026d:trace_of_g} we may write, for
$g = \h f$,
\begin{eqnarray}
\lefteqn{ \hspace*{-2.5cm} \int \limits _{\R ^6}  \Big ( \tr [(\h f)(\dsym _2)] - \tr [(\h f)(\dsym _1)]\Big )
\, dx \, d\x = 2 \int \limits _{\R ^6} \Big ( (\h f)(p_{+,2}(x,\x )) - (\h f)(p_{+,1}(x,\x )) } \nonumber \\
& & \hspace*{1cm}+ (\h f)(p_{-,2}(x,\x )) - (\h f)(p_{-,1}(x,\x )) \Big ) \, dx\, d\x = 2 \langle \o , \h f \rangle ,
\label{ZCMM2026d:robert_trace_simplified}
\end{eqnarray}
where the distribution $\o $ is defined in \eqref{ZCMM2026d:omega_def}. By Theorem~\ref{ZCMM2026d:weyl_trace_formula} applied to $g=\chi f_j\in C^\infty_0(\R)$ and by \eqref{ZCMM2026d:robert_trace_simplified},
\[
\Tr\big[ (\chi f_j)(\Dl_2)-(\chi f_j)(\Dl_1)\big]
=
2(2\pi\hbar)^{-3}\langle\o,\chi f_j\rangle
+
O\bigl(C(f_j)\hbar^{-2}\bigr).
\]
Substituting this into \eqref{ZCMM2026d:step_local}, with the resonance sums restricted to $W$ as in Step~5, and taking absolute values, we obtain
\begin{equation} 
\left|  \sum_{z_k\in\Res(\Dl_2)\cap W} f_j(z_k)  -  \sum_{z_k\in\Res(\Dl_1)\cap W} f_j(z_k)
\right|  \geq c_1\hbar^{-3} \bigl|\langle \o,\chi f_j\rangle\bigr| \quad -c_2e^{-C_0\lambda_j}\hbar^{-3} -c_3(f_j)\hbar^{-2},
\label{ZCMM2026d:count_pre}
\end{equation} 
where $c_1=2(2\pi)^{-3}$.  We now choose $j$ sufficiently large so that
\eqref{ZCMM2026d:FBI_lower} holds,
\[
   \kappa_j\le \frac{C_0}{2},
   \qquad
   |\alpha_j-E_0|\le \frac b2,
   \qquad
   \frac12\le\beta_j\le2,
\]
and
\[
   c_2e^{-C_0\lambda_j}
   \le
   \frac12c_1e^{-\kappa_j\lambda_j}.
\]
Such a choice is possible because $\kappa_j\to0$,
$\lambda_j\to\infty$, and hence
\[
   e^{-(C_0-\kappa_j)\lambda_j}\longrightarrow0.
\]
With this $j$, and therefore $f_j$, now fixed, choose $\hbar$ so
small that
\[
   0<\hbar\leq    \hbar_j:=    \frac{c_1}{4c_3(f_j)}    e^{-\kappa_j\lambda_j}.
\]
Then
\[
   c_3(f_j)\hbar^{-2}    \leq    \frac14c_1    e^{-\kappa_j\lambda_j}\hbar^{-3}.
\]
Using \eqref{ZCMM2026d:FBI_lower} in \eqref{ZCMM2026d:count_pre}, we conclude that
\begin{equation}
\left|  \sum_{z_k\in\Res(\Dl_2)\cap W} f_j(z_k)
 -  \sum_{z_k\in\Res(\Dl_1)\cap W} f_j(z_k)
\right|
\geq \frac{c_1}{4}
e^{-\kappa_j\lambda_j}\hbar^{-3},
\qquad 0<\hbar\le\hbar_j.
\label{ZCMM2026d:count_lower}
\end{equation}
\noindent
\emph{Step 7: counting.} Every term on the left-hand side of \eqref{ZCMM2026d:count_lower} is
bounded by $M_j$ of \eqref{ZCMM2026d:f_big}, the resonances being counted with multiplicity
and lying in $\{ \im z \le 0\}$.  Hence
\[
\sum_{i=1}^2 \#\bigl(\Res(\Dl_i)\cap U\bigr) \geq \sum_{i=1}^2 \#\bigl(\Res(\Dl_i)\cap W\bigr)
\geq \frac{c_1}{4M_j} e^{-\kappa_j\lambda_j}\hbar^{-3} =:C(U)\hbar^{-3},
\]
for $0<\hbar\le\hbar_j$, with
\[
C(U) = \frac{c_1}{4} \exp\left[ -\lambda_j\left( \kappa_j+\frac{a^2}{2} \right)
\right] >0.
\]
This proves Theorem~\ref{ZCMM2026d:thm_rel} near $E_0$; the statement near $-E_0$ 
follows from Lemma~\ref{ZCMM2026d:susy_lemma}, cf.\ Remark~\ref{ZCMM2026d:remark_negative}.
\end{proof}

\begin{remark}
The proof above follows \cite{jkmm14} in outline, and we record the points at which it differs, 
for the reader who compares the two. \emph{(i)} The two box parameters must
satisfy $a \lesssim b^2$, and not merely $a/b \ll 1$: on the vertical edges of $\pd \O $
the exponent in \eqref{ZCMM2026d:wave_packet_size} equals $\l _j [2a+2a^2-\frac{b^2}{8}]$,
which for $a=b/10$ is $\approx \l _j (\frac b5 - \frac{b^2}{8})>0$. \emph{(ii)} The
passage from sums over $\O $ to sums over $W$ in Step 5 uses the upper bound of
Theorem~\ref{ZCMM2026d:upper_bound_thm} explicitly. \emph{(iii)} The restriction
$\im z \le 0$ in the error term of Theorem~\ref{ZCMM2026d:local_trace_formula} is essential
and not merely convenient: on the upper edge $\{ \im z = a\}$ of $\pd \O $ the wave packet
$f_j$ is exponentially \emph{large}, by \eqref{ZCMM2026d:wave_packet_size}. A fourth difference,
concerning the comparison potential, is discussed in Appendix~\ref{ZCMM2026d:appendix_comparison}.
\label{ZCMM2026d:remark_differences}
\end{remark}

\begin{proof}[Proof of Theorem~\ref{ZCMM2026d:thm_abs}]
Put $\ve ' := \ve /4$ and fix the distortion parameter and the comparison neighbourhood as in
Proposition~\ref{ZCMM2026d:comparison_potential}: a parameter $\th _0 \in D_\ve ^+$ with
$\im \th _0 > 0$ and $|\th _0| \le \ve '/4$, and a complex neighbourhood
$U' \Subset \C $ of $E_0$ with
$$
d(\th _0) := \dist \big ( \overline{U'} , \G _{\th _0}\big ) > 0 ,
$$
which is possible because $E_0>1$ is real while $\th _0$ is not, so that
$\dist (E_0, \G _{\th _0})>0$. This is no restriction: by
Proposition~\ref{ZCMM2026d:khochmanprop}(iii) the resonances themselves do not depend on the
distortion family, and the two conditions on $\th _0$ serve only the construction of the
appendix --- the second of them is what places the dilated ray $(1+\th _0)\R ^3$ inside the
cone $C_{\ve ',0}$ on which the comparison potential is uniformly small.

By Proposition~\ref{ZCMM2026d:comparison_potential} of
Appendix~\ref{ZCMM2026d:appendix_comparison} there exists a real-valued $v_1$, entire on
$\C ^3$, such that
\begin{itemize}
\item[(P1)] the pair $(v_1,v_2)$ satisfies Assumption $(\Ab _\d )$ with some $\d >3$, and
$v_1$ is entire with $R_0=0$;
\item[(P2)] $\sup _{\R ^3}|v_1| < \sup v_2$, so that $E_0 = \max \supp (\m )$;
\item[(P3)] $\Res (\Dl _1) \cap U' = \emptyset $.
\end{itemize}
In particular $\m \ne 0$, and $E_0 = \max \supp (\m )$ is a boundary point of $\supp (\m )$
with $E_0>1$, so that Theorem~\ref{ZCMM2026d:thm_rel} applies to the pair $(\Dl _1, \Dl _2)$.

Let $U$ be a complex neighbourhood of $E_0$; shrinking it we may assume $U \subset U'$. By
(P1), (P2) and Theorem~\ref{ZCMM2026d:thm_rel} there is $C(U)>0$ with
$$
\sum _{i=1}^2 \# \big ( \Res (\Dl _i)\cap U \big ) \ge C(U)\hbar ^{-3}
$$
for $\hbar $ small, while the term with $i=1$ vanishes by (P3). Hence
$\# (\Res (\Dl _2)\cap U) \ge C(U)\hbar ^{-3}$. The statement near $-E_0$ follows from
Lemma~\ref{ZCMM2026d:susy_lemma}.
\end{proof}

\subsection{Resonances are born at trapping energies: a trapping--nontrapping dichotomy}
\label{ZCMM2026d:trapping_section}
We finally discuss the dynamical meaning of the energy appearing in the single-potential result and its relation to the resonance-free theorem of 
Chen--Melgaard~\cite{chen_melgaard26}.  This discussion concerns Theorem~\ref{ZCMM2026d:thm_abs}: no corresponding trapping assertion is made
here for an arbitrary boundary point of $\supp(\m)$ occurring in the relative Theorem~\ref{ZCMM2026d:thm_rel}.

For the single-potential theorem, the distinguished positive energy is
$E_0=1+\sup_{x\in\R^3}v(x)$.  This energy is necessarily trapping.  More precisely, 
the maximum of the mass function gives an equilibrium of the Hamilton flow of the 
positive energy branch.

\begin{proposition}
\label{ZCMM2026d:prop_trapping}
Let $v\in C^\infty(\R^3;\R)$ satisfy $v(x)\to0$ as $|x|\to\infty$ and
$\sup v>0$, and put
\[
E_0=1+\sup_{x\in\R^3}v(x) =\sup_{x\in\R^3} m(x).
\]
Then the supremum is attained at some $x_0\in\R^3$, and:
\begin{enumerate}
\item $(x_0,0)\in p_+^{-1}(E_0)$ is an equilibrium of the Hamilton flow of $p_+$;
\item
$E_0$ is not a nontrapping energy in the sense of \cite[Definition~2.3]{chen_melgaard26};
\item no escape function in the sense of
\cite[Definition~2.4]{chen_melgaard26} exists at energy $E_0$: for every $G\in C^\infty(T^*\R^3;\R)$ one has
\[
\{p_+,G\}(x_0,0)=0,
\]
so that
\[
\{p_+,G\}\geq C>0  \qquad\text{on }p_+^{-1}(E_0)
\]
is impossible.
\end{enumerate}
The analogous statements hold for $p_-$ at the energy $-E_0$.
\end{proposition}

\begin{proof}
Since $v(x)\to0$ as $|x|\to\infty$ and $\sup v>0$, the supremum of $v$ is attained at some $x_0\in\R^3$.  At such a point $\nabla v(x_0)=0$.  Since
\[
 p_+(x,\xi)=\bigl(m(x)^2+|\xi|^2\bigr)^{1/2},
        \qquad m(x)=1+v(x),
\]
we have
\[
\nabla_\xi p_+(x,\xi)  =\frac{\xi}{p_+(x,\xi)}, \qquad  \nabla_x p_+(x,\xi) =\frac{m(x)\nabla v(x)}{p_+(x,\xi)}.
\]
Consequently,
\[
\nabla_\xi p_+(x_0,0)=0,  \qquad  \nabla_x p_+(x_0,0)=0,  \qquad  p_+(x_0,0)=m(x_0)=E_0.
\]
Thus the Hamilton vector field $H_{p_+}$ vanishes at $(x_0,0)$, so the trajectory through this point is constant.  It therefore cannot escape from compact subsets of $p_+^{-1}(E_0)$, which proves that $E_0$ is trapping and gives (i)--(ii).

Moreover, for every smooth real-valued $G$,
\[
 \{p_+,G\}(x_0,0)  =H_{p_+}G(x_0,0)=0.
\]
Hence no escape function can satisfy a strictly positive lower bound for $\{p_+,G\}$ on the whole energy surface $p_+^{-1}(E_0)$, proving (iii).

Since $p_-=-p_+$, the same argument applies to the energy $-E_0$.
\end{proof}

Proposition~\ref{ZCMM2026d:prop_trapping} places the energy singled out by
Theorem~\ref{ZCMM2026d:thm_abs} outside the nontrapping regime of
Chen--Melgaard~\cite[Theorem~5.11]{chen_melgaard26}.  In particular,
that resonance-free theorem cannot be applied with $\lambda_0=E_0$:
its nontrapping hypothesis necessarily fails there.  Thus there is no
conflict between the resonance-free strip obtained in \cite[Theorem~5.11]{chen_melgaard26} 
at nontrapping energies and the lower bound of Theorem~\ref{ZCMM2026d:thm_abs}, 
which produces at least $C\hbar^{-3}$ resonances in every complex neighbourhood of the 
trapping energy $E_0$.

The two results do, however, combine to the right of $E_0$.  If every
energy in a compact interval
\[
[E_0+\zeta',E_0+\zeta],        \qquad 0<\zeta'<\zeta,
\]
is nontrapping and satisfies the hypotheses of \cite[Theorem~5.11]{chen_melgaard26}, then the resonance-free domains obtained there can be chosen uniformly on this interval.  This confines
all resonances in the corresponding box to lie either to the left of
$E_0+\zeta'$ or below a logarithmic strip.

\begin{theorem}
\label{ZCMM2026d:cor_wedge}
Let $\Dl_2=\Dl_0+\beb v_2$ be as in
Theorem~\ref{ZCMM2026d:thm_abs}, and put
$E_0:=1+\sup_{x\in\R^3}v_2(x)$.  Let $0<\zeta'<\zeta$, and assume that every energy in the interval
$I=[E_0+\zeta',E_0+\zeta]$ is nontrapping and satisfies the hypotheses of \cite[Theorem~5.11]{chen_melgaard26}. Then 
there exist constants $C>0$ and $\hbar_0\in(0,1]$ such that,
for $0<\hbar<\hbar_0$, the following hold:
\newline
\noindent
(i) Uniform resonance-free strip, i.e., $\Res(\Dl_2)\cap\bigl(I-i[0,C\hbar\log\hbar^{-1}]\bigr)=\emptyset$.
\newline
\noindent
(ii) Localisation, i.e.,  every resonance $z$ of $\Dl_2$ satisfying
\[
E_0-\zeta<\re z<E_0+\zeta,        \qquad        -\zeta<\im z<0,
\]
satisfies
\[
        \re z\leq E_0+\zeta'        \qquad\text{or}\qquad        \im z\leq -C\hbar\log\hbar^{-1}.
\]
\end{theorem}

\begin{proof}
Let $\lambda\in I$.  By hypothesis $\lambda$ is nontrapping and satisfies
the assumptions of \cite[Theorem~5.11]{chen_melgaard26}.  Hence there
exist $\delta_\lambda>0$, $C_\lambda>0$ and $\hbar_\lambda>0$ such that
\[
  \Res(\Dl_2)\cap \Bigl( [\lambda-\delta_\lambda,\lambda+\delta_\lambda]
        -i[0,C_\lambda\hbar\log\hbar^{-1}] \Bigr)        =\emptyset,         \qquad         0<\hbar<\hbar_\lambda.
\]
The intervals
\[
 (\lambda-\delta_\lambda,\lambda+\delta_\lambda), \qquad \lambda\in I,
\]
form an open cover of the compact interval $I$.  Choose a finite
subcover corresponding to $\lambda_1,\ldots,\lambda_N$, and put
\[
 C:=\min_{1\leq k\leq N}C_{\lambda_k}, \qquad \hbar_0:= \min\{\hbar_{\lambda_1},\ldots,\hbar_{\lambda_N},1\}.
\]
Then
\[
 \Res(\Dl_2)\cap\Bigl( I-i[0,C\hbar\log\hbar^{-1}]  \Bigr) =\emptyset, \qquad  0<\hbar<\hbar_0.
\]
Now let $z$ be a resonance of $\Dl_2$ in
\[
(E_0-\zeta,E_0+\zeta)-i(0,\zeta)
\]
and suppose that $\re z>E_0+\zeta'$.  Then
\[
\re z\in(E_0+\zeta',E_0+\zeta)\subset I.
\]
If also
\[
\im z>-C\hbar\log\hbar^{-1},
\]
then, since $\im z<0$,
\[
z\in I-i[0,C\hbar\log\hbar^{-1}],
\]
contradicting the resonance-free property above.  Therefore
\[
\im z\leq-C\hbar\log\hbar^{-1},
\]
which proves the assertion.
\end{proof}

\noindent
\textit{Acknowledgement}. The authors would like to thank the China Scholarship Council (CSC) 
for the award of a PhD studentship to the first author (CSC grant no.\ 202106740004). The second author also acknowledges financial support from the Dr Perry James Browne Sussex Research 
Centre on Mathematics and its Applications.
\newline

\noindent
\textit{Data availability statement}. Data sharing is not applicable to this article as no new data were created or analyzed in this study. 


%
\appendix
\section{The comparison potential}
\label{ZCMM2026d:appendix_comparison}
In this appendix we construct the potential $v_1$ used in the proof of Theorem~\ref{ZCMM2026d:thm_abs}. 
A comparison potential is also used in \cite[Proof of Theorem 4.2]{jkmm14}; in the present construction
the mollifier has \emph{constant} width and vanishing moments up to order three, which is what produces the 
required decay of the difference $v_2-v_1$, and the estimate in the sector is obtained by a contour deformation
which stays inside the cone in which $v_2$ is analytic.

Throughout this appendix we assume directly that $v_2\in C^\infty(\R^3;\R)$ is bounded, that it 
extends holomorphically to an exterior cone $C_{\ve,R_2}$ for some $\ve\in(0,\frac12)$ and 
$R_2\ge0$, and that for some $\r>0$ $|v_2(z)|\le C\kum{z}^{-\r}$ there. We also assume 
$\inf_{x\in\R^3}(1+v_2(x))>0$ and $\sup_{x\in\R^3}v_2(x)>0$. These are precisely the 
individual hypotheses needed on $v_2$ before the comparison potential has been constructed; 
Assumption $(\Ab_\d)$ will only be invoked below for the pair $(v_1,v_2)$. 

Put $\ve':=\ve/4$. First split off the exterior-vanishing case. If $v_2(x)=0$ for all $|x|\ge R_*$ for 
some $R_*$, then uniqueness of the holomorphic continuation implies that the exterior analytic 
continuation also vanishes after increasing the exterior radius. In that case we take $v_1\equiv0$; the conclusions 
of Proposition~\ref{ZCMM2026d:comparison_potential} below are immediate, with any $\d>3$, and no 
mollification is needed. Henceforth suppose that $\ve_0(R):=\sup_{|y|\ge R}|v_2(y)|>0$ for every 
sufficiently large $R$. Choose $R\ge\max(R_2,1)$ so large that $0<\ve_0:=\ve_0(R)<e^{-2}$, and 
put $R':=8R$, $\L_0:=(\log(1/\ve_0))^{1/2}$ and $\l:=\ve'R'/\L_0$. 

We use the explicitly normalised entire mollifier
\[
Q(z):=\frac{1}{2(2\pi)^{3/2}}\bigl(5-z\cdot z\bigr)e^{-z\cdot z/2},
\qquad z\cdot z:=\sum_{j=1}^3 z_j^2,
\]
and set $Q_\l(z):=\l^{-3}Q(z/\l)$. On $\R^3$, $Q$ is real and even, $\int_{\R^3}Q(u)\,du=1$, and
$\int_{\R^3}u^\a Q(u)\,du=0$ for $1\le|\a|\le3$. Indeed, the odd moments vanish by parity, the mixed quadratic moments vanish by symmetry, and the coefficient $5$ is chosen so that $\int u_1^2Q(u)\,du=0$. The use of $z\cdot z$, rather than $|z|^2$, makes $Q$ entire and gives
\[
|Q(p+iq)|\le C\bigl(1+|p|^2+|q|^2\bigr)
\exp\!\left(\frac{|q|^2-|p|^2}{2}\right),\qquad p,q\in\R^3.
\]
With $\h \in \ccs (\R ^3)$, $\h = 1$ on $B(0,1)$, $\supp \h \subset B(0,2)$, put
$$
v_{2,R} := \big ( 1 - \h (\cdot /R)\big ) v_2 , \qquad
v_1(z) := \int \limits _{\R ^3} Q_\l (z-y)\, v_{2,R}(y)\, dy , \quad z\in \C ^3 .
$$
For $z=x+ib$ and $y\in\R^3$, the preceding estimate on $Q$ gives a Gaussian factor 
$\exp(-( |x-y|^2-|b|^2)/(2\l^2))$ times a polynomial. Hence the defining integral and all of its 
$z$-derivatives converge locally uniformly in $z$, so $v_1$ is entire. On $\R^3$ it is real valued 
and, by differentiation and integration by parts, $\partial^\a v_1=Q_\l\ast\partial^\a w$. Note that 
$v_{2,R}=0$ on $|y|\leq R$, $v_{2,R}=v_2$ on $|y|\ge2R$, $\| v_{2,R} \|_\infty\le\ve_0$ and, by 
Lemma~\ref{ZCMM2026d:cauchy_estimates} together with the fixed cut-off, $|\partial^\a v_{2,R} (y)|\leq C_{\a,R}\kum{y}^{-\r-|\a|}$.

\begin{lemma}
$\| v_1 \| _{L^\infty (\R ^3)} \le C\ve _0$, and for every $\a $ and $|x| \ge 2R$,
$$
\big |\partial ^\a (v_2-v_1)(x)\big | \le C_\a \Big ( \l ^4 \kum{x}^{-\r -4-|\a |}
+ e^{-|x|^2/(32\l ^2)}\Big ) \le C_\a ' \kum{x}^{-\r -4 -|\a |} .
$$
\label{ZCMM2026d:comparison_real}
\end{lemma}

\begin{proof}
The first assertion is
$\| Q_\l \ast w\| _\infty \leq \| Q\| _{L^1}\|  v_{2,R}\| _\infty $. For the second, note that
$v_2=v_{2,R}$ near $x$ when $|x|\ge 2R$, so it suffices to estimate
$W - Q_\l \ast W$ with $W := \partial ^\a v_{2,R}$. Split
\begin{align*}
\big ( Q_\l \ast W - W\big )(x) &= \int \limits _{|u|\le |x|/4}Q_\l (u)
\big [ W(x-u)-W(x)\big ]\, du \\
&\phantom{=} + \int \limits _{|u|>|x|/4}Q_\l (u)\big [ W(x-u)-W(x)\big ] \, du .
\end{align*}
In the first integral expand $W$ by Taylor to order four about $x$; the terms of order
one, two and three integrate to zero by the moment conditions on $Q$ (up to a tail which
is absorbed in the second integral), and the remainder is bounded by
$C|u|^4 \sup _{|s|\le |x|/4}|\nabla ^4 W(x-s)| \le C|u|^4 \kum{x}^{-\r -4-|\a |}$,
since $|x-s| \ge \frac34 |x|$; multiplying by $Q_\l $ and integrating gives
$C\l ^4 \kum{x}^{-\r -4-|\a |}$. The second integral is bounded by
$2\| W\| _\infty \int _{|u|>|x|/4}|Q_\l (u)|\, du \le Ce^{-|x|^2/(32\l ^2)}$.
\end{proof}

\begin{lemma}
There are $C,c>0$ such that for all $z=x+ib$ with $|x|\ge R'$ and $|b|\le\ve'|x|$,
\begin{equation}
|v_1(z)-v_2(z)| \le C\l ^4\kum{z}^{-\r-4}+Ce^{-c|x|^2/\l^2}.
\label{ZCMM2026d:sector_estimate}
\end{equation}
Moreover, for $|x|\le R'$ and $|b|\le\ve'|x|$ one has
\[
|v_1(z)|\le C\Bigl(1+\frac{|b|^2}{\l^2}\Bigr)
 e^{|b|^2/(2\l^2)}\ve_0
\le C(1+\L_0^2)\ve_0^{1/2}.
\]
Consequently, after increasing $C$ if necessary,
\begin{equation}
\sup_{z\in C_{\ve',0}}|v_1(z)|\le
\eta_R:=C\Bigl[(1+\L_0^2)\ve_0^{1/2}+(R')^{-\r}
+\ve_0^{\,c/\ve'^2}\Bigr]\longrightarrow0
\qquad(R\to\infty).
\label{ZCMM2026d:comparison_smallness}
\end{equation}
\label{ZCMM2026d:comparison_sector}
\end{lemma}

\begin{proof}
\emph{The far part.} Let $\tilde{\p }\in \cs ([0,\infty );[0,1])$ with $\tilde{\p }=1$ on
$[0,\frac14]$ and $\tilde{\p }=0$ on $[\frac12,\infty )$, and put
$\p (t) := \tilde{\p }(|t-x|/|x|)$, so that $|\nabla \p | \le C/|x|$. Consider the
contours
$$
\G ^\s := \big \{ y = t + i\s \p (t) b \, : \, t \in \R ^3 \big \}, \qquad \s \in [0,1],
$$
so that $\G ^0 = \R ^3$. If $\p (t)\ne 0$ then $|t-x|<|x|/2$, hence
$|t| > \frac12 |x| \ge 4R \ge 2R$ and $|t|>R_2$, so that $v_{2,R}=v_2$ on the swept region and
$v_2$ is holomorphic there; moreover
$|\im y| \le |b| \le \ve '|x| \le 2\ve '|t| = \frac{\ve }{2}|t| < \ve |\re y|$, so the
swept region lies in $C_{\ve ,R_2}$. The integrand
$Q_\l (z-y)v_{2,R}(y)\, dy_1 \wedge dy_2\wedge dy_3$ is therefore a closed holomorphic form
there, with Gaussian decay in $|t|$, and Stokes' theorem gives
$v_1(z) = \int _{\G ^1}Q_\l (z-y)v_{2,R}(y)\, dy$.

On the part of $\G ^1$ where $|t-x| \ge |x|/4$ we have $|\im (z-y)| = |b||1-\p (t)|
\le \ve '|x|$ and $|\re (z-y)| = |x-t| \ge |x|/4$, so that
\begin{align*}
\big | Q\big ( (z-y)/\l \big )\big | &\le C\Big ( 1 + \frac{|z-y|^2}{\l ^2}\Big )
\exp \Big ( \frac{|\im (z-y)|^2 - |\re (z-y)|^2}{2\l ^2}\Big ) \\
&\le C' e^{-c|x-t|^2/\l ^2}e^{-c|x|^2/\l ^2},
\end{align*}
because $\ve '^2|x|^2 - \frac{1}{16}|x|^2 \le -\frac{1}{32}|x|^2$ for
$\ve ' \le \frac18$. The Jacobian of the contour is
$\det (\Idb + ib \ot \nabla \p ) = 1 + \Oc (\ve ')$.
The bound $\|v_{2,R}\|_{L^\infty(\R^3)}\le\ve_0$ is only a real-axis bound and should not be used 
after the contour has been moved. On the swept region, however, $v_{2,R}=v_2$ and the exterior analytic 
decay gives $|v_{2,R}(y)|=|v_2(y)|\le C\kum{t}^{-\r}\leq C$. Together with the Gaussian estimate 
(whose polynomial factor is absorbed by a slight weakening of the Gaussian exponent), 
this shows that the part $|t-x|\ge|x|/4$ contributes $\Oc(e^{-c|x|^2/\l^2})$. 

On the part where $|t-x| \le |x|/4$ we have $\p \equiv 1$, the Jacobian equals one, and
$z - y = x-t$ is \emph{real}: the Gaussian growth has disappeared. Hence
$$
v_1(z) = \int \limits _{|t-x|\le |x|/4}Q_\l (x-t)F(t)\, dt
+ \Oc \big ( e^{-c|x|^2/\l ^2}\big ), \qquad F(t):= v_2(t+ib).
$$
By Lemma~\ref{ZCMM2026d:cauchy_estimates}, $|\nabla ^4 F(t)| \le C\kum{z}^{-\r -4}$ for
$|t-x| \le |x|/4$. Let $P_3$ be the Taylor polynomial of $F$ at $x$ of degree three. Then 
$|F(t)-P_3(t)|\le C|t-x|^4\kum{z}^{-\r-4}$ on $|t-x|\le|x|/4$. The full-space moments of 
$Q_\l$ give $\int_{\R^3}Q_\l(x-t)P_3(t)\,dt=P_3(x)=F(x)$. Replacing the full-space integral by $|t-x|\le|x|/4$ produces only a Gaussian tail, because $|x|/\l\gtrsim\L_0/\ve'$. Hence
\[
\left|\int_{|t-x|\le|x|/4}Q_\l(x-t)F(t)\,dt-F(x)\right|
\le C\l^4\kum{z}^{-\r-4}+Ce^{-c|x|^2/\l^2},
\]
which is \eqref{ZCMM2026d:sector_estimate}, since $F(x)=v_2(z)$. 

\emph{The bounded part.} Write $z=x+ib$ and change variables $p=(x-y)/\l$. The entire-kernel estimate above gives
\begin{align*}
|v_1(z)|
&\le C\ve_0 e^{|b|^2/(2\l^2)}
\int_{\R^3}\Bigl(1+|p|^2+\frac{|b|^2}{\l^2}\Bigr)e^{-|p|^2/2}\,dp \\
&\le C\ve_0\Bigl(1+\frac{|b|^2}{\l^2}\Bigr)e^{|b|^2/(2\l^2)}.
\end{align*}
For $|x|\le R'$ and $|b|\le\ve'|x|$ we have $|b|/\l\le\ve'R'/\l=\L_0$, and therefore
$|v_1(z)|\le C(1+\L_0^2)\ve_0^{1/2}$. Since $(1+\L_0^2)\ve_0^{1/2}\to0$, this is the smallness needed below. For $|x|\ge R'$, \eqref{ZCMM2026d:sector_estimate}, $|v_2(z)|\le C(R')^{-\r}$, $\l^4(R')^{-\r-4}\le C(R')^{-\r}$ and $e^{-c(R')^2/\l^2}=\ve_0^{c/\ve'^2}$ give \eqref{ZCMM2026d:comparison_smallness}.
\end{proof}

\begin{proposition}
Let $v_2$ satisfy the individual hypotheses stated at the beginning of this appendix and put 
$E_0:=1+\sup v_2>1$. Fix a distortion parameter $\th_0\in D_\ve^+$ such that $0<\im\th_0$ and 
$|\th_0|\le\ve'/4$, and let $U'\Subset\C$ be a complex neighbourhood of $E_0$ such that
\[
d(\th_0):=\dist(\overline{U'},\G_{\th_0})>0.
\]
Then, for $R$ sufficiently large, the comparison potential $v_1$ constructed above is real valued on $\R^3$ and entire on $\C^3$, and it has the following properties: 
\begin{itemize} 
\item[(i)] $|v_1(z)|=\Oc(\kum{z}^{-\r})$ on $C_{\ve',0}$, $\|v_1\|_{L^\infty(\R^3)}\le C\ve_0$, and the uniform 
sector smallness satisfies $\sup_{z\in C_{\ve',0}}|v_1(z)|\le\eta_R\to0$ as $R\to\infty$, with $\eta_R$ 
given by \eqref{ZCMM2026d:comparison_smallness}. 

\item[(ii)] With $\d:=\r+4>3$ one has
\[
|v_2(z)-v_1(z)|=\Oc(\kum{z}^{-\d}),\qquad z\in C_{\ve',R'},
\]
and, on the real axis, for every multi-index $\a$,
\[
|\partial^\a(v_2-v_1)(x)|\le C_{\a,R}\kum{x}^{-\d-|\a|},\qquad |x|\ge2R.
\]
For $R$ large enough, $\inf(1+v_1)>0$; hence the pair $(v_1,v_2)$ satisfies Assumption $(\Ab_\d)$, for example on the common cone $C_{\ve',R'}$. 
\item[(iii)] $\sup_{\R^3}|v_1|<\sup v_2$ and, for the relative distribution $\m$ associated with the pair $(v_1,v_2)$,
$E_0=\max\supp(\m)$. 
\item[(iv)] $\Res(\Dl_1)\cap U'=\emptyset$.
\end{itemize}
\label{ZCMM2026d:comparison_potential}
\end{proposition}

\begin{proof}
We treat the nontrivial case $\ve_0(R)>0$; the exterior-vanishing case $v_1\equiv0$ was already disposed of above. 
For $|\re z|\ge R'$, \eqref{ZCMM2026d:sector_estimate} and the decay of $v_2$ give
\[
|v_1(z)|\le C\kum{z}^{-\r}+C\l^4\kum{z}^{-\r-4}+Ce^{-c|\re z|^2/\l^2}
=\Oc(\kum{z}^{-\r}).
\]
On the bounded part of $C_{\ve',0}$ the function $v_1$ is entire and therefore bounded. This proves the decay assertion in (i). The real-axis estimate $\|v_1\|_\infty\le C\ve_0$ is Lemma~\ref{ZCMM2026d:comparison_real}, while the uniform complex-sector smallness is exactly \eqref{ZCMM2026d:comparison_smallness}. 

Next put $\d=\r+4$. On $C_{\ve',R'}$, \eqref{ZCMM2026d:sector_estimate} gives
\[
|v_2-v_1|\le C\l^4\kum{z}^{-\d}+Ce^{-c|\re z|^2/\l^2}.
\]
Once $R$ has been fixed, $\l$ is fixed, and the Gaussian is bounded by $C_{N,R}\kum{z}^{-N}$ for every $N$ on this exterior cone; taking $N=\d$ proves the required complex estimate. The real derivative estimate is Lemma~\ref{ZCMM2026d:comparison_real}. Since $\|v_1\|_\infty\to0$ and $\inf(1+v_2)>0$ by hypothesis, for large $R$
\[
\min\{\inf(1+v_1),\inf(1+v_2)\}>0.
\]
Together with the individual decay of $v_1$ and $v_2$, this verifies Assumption $(\Ab_\d)$ for the pair and proves (ii). 

Since $\|v_1\|_\infty\to0$ and $\sup v_2>0$, we may further increase $R$ so that $\sup_{\R^3}|v_1|<\sup v_2$. Hence
$\supp\m_{+,1}\subset[1,1+\sup|v_1|]$, while $E_0=1+\sup v_2$ lies strictly to the right of that interval. Because $v_2(x)\to0$ and $\sup v_2>0$, the maximum of $v_2$ is attained at some $x_0$. By continuity, every interval $(E_0-\varsigma,E_0]$ has positive $\m_{+,2}$-mass. Thus $E_0\in\supp\m_{+,2}$, and \eqref{ZCMM2026d:mu_restricted} shows that $\m=\m_{+,2}$ in a neighbourhood of $E_0$. Therefore $E_0=\max\supp\m$, proving (iii). 

Finally, because $|\th_0|\le\ve'/4$, the dilated ray $(1+\th_0)\R^3$ lies in the cone $|\im z|\le\ve'|\re z|$. Since $v_1$ is entire and has the decay from (i), the distortion for $\Dl_1$ may be taken to be the pure dilation 
$\f_{\th_0}(x)=(1+\th_0)x$. Set
\[
M_{\th_0,U'}:=\sup_{z\in U'}\| (\Dl_{0,\th_0}-z)^{-1}\|.
\]
The separation $d(\th_0)>0$ and Lemma~\ref{ZCMM2026d:resolvent_lemma} imply $M_{\th_0,U'}<\infty$. By \eqref{ZCMM2026d:comparison_smallness},
\[
\sup_{x\in\R^3}|v_1((1+\th_0)x)|\le\eta_R\longrightarrow0.
\]
Choose $R$ so large that $\eta_R M_{\th_0,U'}<1$. Then, for every $z\in U'$,
\[
\Dl_{1,\th_0}-z
=\Bigl(\Idb+\beb\,(v_1\circ\f_{\th_0})(\Dl_{0,\th_0}-z)^{-1}\Bigr)
(\Dl_{0,\th_0}-z),
\]
and the first factor is invertible by a Neumann series. Hence $\Dl_{1,\th_0}-z$ is invertible throughout $U'$, so 
$\Res(\Dl_1)\cap U'=\emptyset$. This proves (iv). 
\end{proof}

\begin{remark}
The clean order of choices is the following. First fix the analyticity cone and hence $\ve'$. Next choose one non-real distortion parameter $\th_0$ with $0<\im\th_0$ and $|\th_0|\le\ve'/4$, and then choose $U'\Subset\C$ around $E_0$ with $\overline{U'}\cap\G_{\th_0}=\emptyset$. Only after $\th_0$ and $U'$ have been fixed do we take $R$ large. The corrected estimate \eqref{ZCMM2026d:comparison_smallness} gives $\eta_R\to0$ with no need to compare the rates of $\ve_0(R)$ and $R^{-\r}$, and this is exactly what is required to make the Neumann-series norm smaller than one. The fixed non-real distortion angle is essential: at $\th_0=0$, the energy $E_0>1$ lies on the free continuous spectrum and the required resolvent bound is unavailable.
\end{remark}

\end{document}